%% file: main.tex
\documentclass[12pt,reqno]{amsart}
\usepackage[top=1in, bottom=1in, left=1in, right=1in]{geometry}
\usepackage{tikz-cd}
\usepackage{comment}

\input{preamble}

\title{Asymptotic Syzygies of Weighted Projective Spaces}
\author{Boyana Martinova}

\begin{document}

\begin{abstract}
By adapting methods of Ein-Erman-Lazarsfeld, we prove an analogue of the Ein-Lazarsfeld nonvanishing result on asymptotic syzygies for Veronese embeddings, in the setting of weighted projective spaces of the form $\mathbb{P}(1^n,2)$.
\end{abstract}
\maketitle

\vspace{-1.5cm}

\section{Introduction}

The aim of this paper is to describe the nonvanishing asymptotic syzygies of $\mathbb{P}(1^n,2)$. Since Green's pioneering work \cite{Green},  understanding the syzygies of projective varieties has been at the center of much research. The asymptotic picture is well-understood for curves: the syzygies become simpler as the positivity of the embedding line bundle grows. It was believed that a similar story held in higher dimensions; however, Ein and Lazarsfeld \cite{EinLazarsfeld_2012} showed the contrary: ``almost every" entry in the Betti table corresponding to the Veronese embedding $\varphi \colon \mathbb{P}^n \xrightarrow[]{|\mathcal{O}(d)|} \mathbb{P}^N$ is nonzero for $d
\gg 0$. To be precise, the regularity and projective dimension provide preexisting bounds on the size of the Betti table, so the result states that, asymptotically, there are generators in essentially every allowable degree. 
Ein-Erman-Lazarsfeld \cite{EEL_2016} later showed the same nonvanishing result by using only facts about the monomials that generate the Veronese embedding. 

By adapting the monomial method developed in \cite{EEL_2016}, we prove similar nonvanishing results for sufficiently large Veronese embeddings of the weighted projective space $\mathbb{P}(1^n,2)$. Those familiar with weighted projective spaces may be confused by our notation, since $\mathcal{O}(d)$ may fail to be very ample. We instead choose the embedding given by the entire Veronese subring of degree $d$ (see Definition \ref{def:weightedVer}). Our main results are summarized below. 

\begin{introtheorem}\label{Thrm:A}
    Let $R$ be the coordinate ring of $\mathbb{P}(1^n,2)$. Let $R^{(d)}$ be the degree $d$ Veronese subring and $S$ be a graded polynomial ring with a minimal surjection $S \to R^{(d)}$ (See Definition \ref{def:weightedVer}). 
    Let $N:= \dim \Proj(S)$. For each $1 \leq q \leq n+1$, there exist constants $c_q$, $C_q$ such that, for any $d \gg 0$, we have $\beta_{i, i+ q}^S\left(R^{(d)}\right) \neq 0$ for $i$ within the following ranges:
\begin{center}
    \bgroup
    \def\arraystretch{1.15}
    \begin{tabular}{ c | c c c | c c c} 
    Row Index ($q$) & \multicolumn{6}{c}{Corresponding Range of $i$ Values} \\
    \hline
    & \multicolumn{3}{c|}{$d$ even} & \multicolumn{3}{c}{$d$ odd} \\
    1 & 1 & --- &  $N-C_1d^{n-1}$ & 1 & --- & $N-C_1d^{n-1}$ \\
    2 & $c_2 d^{1}$ & --- & $N-C_2 d^{n-2}$ & $c_2d^0$ & --- & $N-C_2d^{n-2}$  \\ 
    $\vdots$ & &$\vdots$ & & & $\vdots$ & \\
    $n-1$ & $c_{n-1}d^{n-2}$ & --- & $N-C_{n-1}d^1$ &  $c_{n-1}d^{n-3}$ & --- & $N-C_{n-1}d^1$ \\
    $n$ & $c_nd^{n-1}$ & --- & $N-n$ & $c_nd^{n-2}$ & --- & $N- n - (n \mod 2)$ \\
    $n+1$ & &$\emptyset$& & $c_{n+1}d^{n-1}$ & --- & $N-n$ \\
    \end{tabular}
    \egroup
\end{center}
\end{introtheorem}

Regularity computations confirm that the Betti table is zero in rows $q > n$ for $d$ even (Lemma \ref{lem:dEvenGens}) and $q > n + 1$ for $d$ odd (Corollary \ref{cor:lastRow}), so the above theorem classifies all rows of $\beta^S(R^{(d)})$.

Theorem \ref{Thrm:A} is a corollary of the following, more precise, result. 

\begin{introtheorem}
\label{thrm:preciseBoundsIntro}
        For $d \gg 0$, the Betti table of the $d$th Veronese embedding $\varphi_d(\mathbb{P}(1^n, 2)) \subseteq \Proj(S)$ (as above) has nonvanishing Betti entries $\beta_{i, i+q}$ for all $F_q(d) \leq i \leq B_q(d)$, where $F_q(d)$ and $B_q(d)$ are defined in the below table. 
        
    \begin{tblr}{
  colspec={Q[2.5cm]Q[2.3cm]Q[10.9cm]},
  hlines,vlines,
  cells={valign=m},
  rows={ht=3\baselineskip},
  row{1}={ht=1.5\baselineskip,font=\bfseries},
    }
        \textbf{Row Index} & \textbf{Parity of } $\bm{d}$ & $\bm{F_q(d)}$ \textbf{ and } $\bm{B_q(d)}$ \\
    
         $q = 1$ & $d=$ even & $F_q(d) = 1$ \newline $B_q(d) =  N - \displaystyle\sum_{b=0}^{\frac{d}{2}} \textstyle{ d-2b+n-2 \choose n-2} + n -2 $ \\
    
         $q = 1$& $d =$ odd & $F_q(d) = 1$ \newline $B_q(d) = N - \displaystyle\sum_{b = 0}^{\frac{d-1}{2}} \textstyle{d-2b + n-2 \choose n-2} + n-3$ \\
    
         $2 \leq q \leq n-1$ & $d =$ even & $F_q(d) = \displaystyle\sum_{b=0}^{\frac{d}{2}} \textstyle{ d-2b+q-1 \choose q-1} - \displaystyle\sum_{b=0}^{\floor{\frac{d-q-2}{2}}} \textstyle{d-2b-3 \choose q-1} -q$ \newline $B_q(d) =  N - \displaystyle\sum_{b=0}^{\frac{d}{2}} \textstyle{ d-2b+n-q-1 \choose n-q-1} + \displaystyle\sum_{b=0}^{\floor{\frac{q}{2}}} \textstyle{-2b+n-1 \choose n-q-1} -q$ \\
    
         $2 \leq q \leq n-1$ & $d =$ odd &  $F_q(d) = \displaystyle\sum_{b=0}^{\frac{d-1}{2}} \textstyle{d - 2b + q - 2 \choose q - 2} - \displaystyle\sum_{b=0}^{\floor*{\frac{d-q-1}{2}}} \textstyle{d-2b-3 \choose q-2} - q + 2$ \newline $B_q(d) = N - \displaystyle\sum_{b = 0}^{\frac{d-1}{2}} \textstyle{d-2b + n-q-1 \choose n-q-1} + \displaystyle\sum_{b=0}^{\floor*{\frac{q}{2}}} \textstyle{-2b+n-1 \choose n-q-1} -q-1$\\ 
    
        $q = n$ & $d =$ even & $F_q(d) = \displaystyle\sum_{b=0}^{\frac{d}{2}} \textstyle{ d-2b+n-1 \choose n-1} - \displaystyle\sum_{b=0}^{\floor{\frac{d-n-2}{2}}} \textstyle{d-2b-3 \choose n-1} -n$ \newline $B_q(d)= N - n$\\
    
        $q = n$ & $d =$ odd & $F_q(d) = \displaystyle\sum_{b=0}^{\frac{d-1}{2}} \textstyle{d - 2b + n - 2 \choose n - 2} - \displaystyle\sum_{b=0}^{\floor*{\frac{d-n-1}{2}}} \textstyle{d-2b-3 \choose n-2} - n + 2$ \newline $B_q(d)= N - n - (n \mod 2)$\\
    
        $q = n+1$ & $d =$ odd & $F_q(d) = \displaystyle\sum_{b=0}^{\frac{d-1}{2}} \textstyle{d - 2b + n-1 \choose n-1} - \displaystyle\sum_{b=0}^{\floor*{\frac{d-n-2}{2}}} \textstyle{d-2b-3 \choose n-1} - n +1$ \newline $B_q(d) = N - n$\\
    \end{tblr}
\end{introtheorem}

Recall the notation introduced by Erman-Yang \cite{ErmanYang_2018}, where they define $\rho_q(M)$ as the ratio of nonzero entries in the $q$th row of the Betti table. In particular, 

\[\rho_q(M):= \frac{\#i\in[0,\operatorname{pdim}(M)] \text{ where } \beta_{i,i+q}(M) \neq 0 }{\operatorname{pdim}(M) +1}.\]
Remark \ref{rmk:asypN} shows that, for $d\gg 0$, $N$ is on the same order of magnitude as $d^n$, so the formulas in Theorem \ref{thrm:preciseBoundsIntro} yield the following corollary. 

\begin{introcor}
    We continue with the hypotheses of Theorem \ref{thrm:preciseBoundsIntro}, and let $M$ be the coordinate ring of the Veronese. Then, we have
    $\rho_q(M) = \begin{cases}
        1 & \text{if } 1 \leq q \leq n\\
        1 & \text{if } q=n+1 \text{ and } d \text{ is odd} \\
        0 & \text{else }
    \end{cases}$
\end{introcor}

This corollary mirrors the original results of Ein-Lazarsfeld \cite{EinLazarsfeld_2012}; it implies that ``almost every" Betti entry in the allowable range is nonzero in this setting. Note that our results do not guarantee there are no nonzero entries outside of the stated bounds. However, for small, computable examples, these bounds seem to correctly locate all nonzero entries, which inspires the following conjecture, joint with Daniel Erman. 
\begin{introconjecture}[Erman-Martinova]
\label{conjecture}
    The bounds presented in Theorem \ref{thrm:preciseBoundsIntro} are sharp. More specifically, $\beta_{i,i+1} =0 $ for $i$ outside of the ranges provided in Theorem \ref{thrm:preciseBoundsIntro}.
\end{introconjecture}

Since Ein and Lazarsfeld's original findings, the study of asymptotic syzygies has been very active, with work that extends their results \cite{bruce2020quantitativebehaviorasymptoticsyzygies, conca2014asymptoticsyzygiesstanleyreisnerrings, park2023asymptoticnonvanishingsyzygiesalgebraic, Raicu_2016, zhou2014effectivenonvanishingasymptoticadjoint, erman2025largealgebraicbettinumbers}, establishes computational and experimental methods \cite{castryck2016computinggradedbettitables, bruce2021syzygiesmathbbp1timesmathbbp1data, bruce2017conjecturescomputationsveronesesyzygies}, and performs syzygy analysis via probability based models \cite{banerjee2021edgeidealserdosrenyirandom, booms2021characteristicdependencesyzygiesrandom, engström2023regularityedgeideals, dochtermann2023randomsubcomplexesbettinumbers, ErmanYang_2018}.

This article follows in the footsteps of a wide range of previous work that aims to extend classical results to the coordinate rings of other toric varieties; namely, nonstandard and multigraded polynomial rings. As will be discussed in Section \ref{section:regularity}, Benson \cite{Benson2004} extended the notion of Castelnuovo-Mumford regularity to the nonstandard graded setting, which played a major role in Symonds’s results on invariant theory in positive characteristic \cite{Symonds2011}. Maclagan and Smith \cite{MaclaganSmith2004} developed an understanding of regularity for multigraded rings, which has been since been further by  \cite{BotbolChardin2017, BruceCrantonHellerSayrafi2021, BruceCrantonHellerSayrafi2022, BESVirtualRes, ChardinHolanda2022, ChardinNemati2020, SidmanVanTuylWang2006}. More broadly, the study of multigraded syzygies has been very active, with work related to Koszul properties \cite{BrownErman2024TateRes, DavisSobieska2025rationalnormalcurvesweighted, EisenbudErmanSchreyer2015}, curves \cite{beLinSyzOCurvesIWProjSp, BrownErman2024LinearStrands, Cobb2024}, truncations \cite{bePositivity2024, CrantonHeller2025, DavisMartinova2025koszulpropertytruncationsnonstandard}, virtual resolutions \cite{BerkeschKleinLoperYang2021, BrownErmanShortRes_2024, HanlonHicksLazarev2024, HaradaNowrooziVanTuyl2022, Yang2021}, and much more \cite{Kingsconjecturebirationalgeometry, BuseChardinNemati2022, BrownSayrafi2024}. Notably, \cite{bruce2019asymptoticsyzygiessettingsemiample} extends the EEL Method to the multigraded setting, by analyzing embeddings of products of projective spaces. 

An overarching theme is that extending results about syzygies from the standard to nonstandard graded setting often requires new methods and perspectives on the classical results. As we will show, extending the methods developed by \cite{EEL_2016} to weighted projective spaces involves overcoming three key obstacles: (1) unlike the original results, it will be necessary to consider multiple monomials for each row of the Betti table, requiring that we carefully track which Betti entries correspond to each monomial and show they overlap, (2) the size of the Betti table depends on the modulus class of $d$ with respect to the degrees of the variables, meaning that multiple cases must be considered in order to classify all Veronese embeddings for a single weighted projective space, and (3) individual monomials may yield nonzero blocks of Betti entries spanning multiple rows, which makes tracking the corresponding nonvanishing syzygies especially nuanced. We focus on weighted projective spaces of the form $\mathbb{P}(1^n,2)$ because this setting allows us to explore the novelties presented in (1) and (2) without the added complication of (3). 

This article is structured as follows: in Section \ref{section:eel}, we detail the original EEL Method and elaborate on the key differences arising from obstacle (1). In Section \ref{section:regularity}, we prove the regularity arguments necessary to bound the number of rows in Betti table for our setting and further outline obstacle (2). Section \ref{section:notation} sets notation for Section \ref{section:MainResults}, which contains our main results. Section \ref{section:MainResults} is partitioned into subsections, first discussing the $d$ odd case in detail through \ref{subsec:FrontResults}, \ref{subsec:BackResults}, and \ref{subsec:Overlap}, then applying the same methods for $d$ even in \ref{subsec:evenVerDeg}. Theorems \ref{Thrm:A} and \ref{thrm:preciseBoundsIntro} are shown in \ref{subsec:MainThrm}, but they heavily rely on results from the previous subsections. Lastly, in Section \ref{section:futureDir}, we explore obstacle (3) by describing some additional challenges that arise when working with different weighted projective spaces.

\subsection*{Acknowledgments}
I would like to express my sincere gratitude to Daniel Erman for his guidance, support, and invaluable feedback throughout this project. I also thank Maya Banks, John Cobb, Caitlin Davis, Jose Israel Rodriguez, and Aleksandra Sobieska for their mentorship and helpful comments. I am grateful to Christine Berkesch, Michael Brown, Juliette Bruce, Gregory Smith, Christin Sum, and many others for their insights and suggestions. I acknowledge the support from the National Science Foundation Grant DMS-2200469, as well as that from the math departments at the University of \Hawaii \ at \Manoa \ and the University of Wisconsin-Madison. Lastly, most explicit computations were performed with the aid of Macaulay2 \cite{M2}.

\section{The Ein-Erman-Lazarsfeld Syzygy Method}
\label{section:eel}
The details of the monomial syzygy method developed by Ein-Erman-Lazarsfeld \cite{EEL_2016} (EEL Method) will be crucial to the main results of this article, so we carefully explain them in this section. The EEL Method discusses Veronese embeddings of $\mathbb{P}^n$, so, for this section only, the corresponding polynomial rings will be standard graded. 

Let $R= k[x_0, \ldots, x_n]$ be a standard graded polynomial ring, so that $\Proj(R) = \mathbb{P}^n$. We consider the Veronese embedding $\varphi_d \colon \mathbb{P}^n \xrightarrow[]{|\mathcal{O}(d)|} \mathbb{P}^N$, and let $S=k[z_0, \ldots, z_N]$ be the polynomial ring corresponding to $\mathbb{P}^N$. Let $M := R^{(d)}$, viewed as an $S$-module. 

Throughout this article, we will use the standard Betti table notation. That is, for a graded $S$-module, $M$, we can construct its minimal free resolution, 
\[\mathcal{F}_\bullet : 0 \longleftarrow M \longleftarrow F_0 \longleftarrow F_1 \longleftarrow \cdots \longleftarrow F_i \longleftarrow \cdots, \] 
and $\beta_{i,i+j}^S(M)$ is the number of degree $i+j$ generators of $F_i$. Alternatively, one can compute a Betti entry by calculating the rank of a specific $\operatorname{Tor}$ group, as $\beta^S_{i,i+j}(M) = \rank_k \left(\operatorname{Tor}_i^S(M, k)_{i+j}\right)$. This equivalence holds for any $S$-module, but we will consider what happens for our specific choice of $M$. 

In this case, $\{x_0^d, \ldots, x_n^d\}$ is a regular sequence on $M$, and we can perform an Artinian reduction. Suppose we order the variables in $S$ so that $z_0, \ldots ,  z_n$ correspond to the elements of this sequence. Then, we can view $\overline{M} :=M/\langle x_0^d, \ldots, x_n^d\rangle$ as an $\overline{S}:= S/\langle z_0, \ldots, z_n \rangle$-module. Since $\{z_0, \ldots, z_n\}$ is a regular sequence, $\mathcal{F}_\bullet \otimes_S S$ has no higher homology, and it is a minimal free resolution of $M \otimes_S \overline{S} = \overline{M}$. In particular, $\beta^S(M) = \beta^{\overline{S}}(\overline{M})$. Therefore, we can instead compute individual Betti entries of $M$ over $S$ by considering $\operatorname{Tor}_i^{\overline{S}}(\overline{M}, k)_{i+j}$. We can explicitly construct these desired $\operatorname{Tor}$ groups as follows. 

First, we resolve $k$ as an $\overline{S}$-module via the Koszul complex: 
\[0 \longleftarrow k \longleftarrow \wedge^0 \overline{S}_1 \otimes_k \overline{S} \longleftarrow \wedge^1 \overline{S}_1 \otimes_k \overline{S}(-1) \longleftarrow \cdots \longleftarrow \wedge^ i \overline{S}_1 \otimes_k \overline{S}(-i) \longleftarrow \cdots.\]

Second, we apply $(- \otimes_k \overline{M})$ to the above complex:
\[0 \longleftarrow \overline{M} \longleftarrow \wedge^0 \overline{S}_1 \otimes_k \overline{M} \longleftarrow \wedge^1 \overline{S}_1 \otimes_k \overline{M}(-1) \longleftarrow \cdots \longleftarrow \wedge^ i \overline{S}_1 \otimes_k \overline{M}(-i) \longleftarrow \cdots.\]

Third, we find the degree $i+j$ component of the homology at $\wedge^i \overline{S}_1 \otimes \overline{M}(-i)$. Since $\overline{M}(-i)_{i+j} = \overline{M}_j$, we can equivalently compute the homology in the strand
\[\wedge^{i-1} \overline{S}_1\otimes \overline{M}_{j+1} \longleftarrow \wedge^i \overline{S}_1\otimes \overline{M}_j \longleftarrow \wedge^{i+1} \overline{S}_1 \otimes \overline{M}_{j-1}.\]

In other words, showing $\beta_{i,i+j} \neq 0$ amounts to finding a nonzero homology element in the above strand. The following lemma, which is adapted from Lemma 2.3 of \cite{EEL_2016}, gives a concrete way of finding such an element. 
\begin{lemma}
\label{Lem:EEL}
    Recall each generator $z_i$ of $\overline{S}$ corresponds to a specific degree $d$ monomial $m_i$ in $\overline{R}$. Let $m$ be a monomial in $\overline{M}$, and let $D(m) \subseteq \{z_{n+1}, \ldots , z_N\}$ denote the subset such that the corresponding $m_i$ divide $m$. Let $A(m) \subseteq \{ z_{n+1}, \ldots , z_N\}$ denote the subset such that the corresponding $m_i$ annihilate $m$ in $\overline{M}$. If $D(m) \subseteq A(m)$, then we have 
    \[\beta_{i, i+\deg(m)}^{\overline{S}} (\overline{M}) \neq 0 \text{ for all } |D(m)| \leq i \leq |A(m)|. \]
\end{lemma}

Consider the following example. 

\begin{example}
\label{ex:EEL}
    Let $R=k[x_0, x_1, x_2]$ and consider $M = R^{(3)} = k[x_0^3, \;x_1^3, \; x_2^3, \;x_0^2x_1, \ldots, x_1x_2^2]$, viewed as an $S = k[z_0, \ldots, z_9]$-module (where $z_0, z_1, z_2$ correspond to the pure powers). We perform an Artinian reduction by the regular sequence $\{x_0^3, x_1^3, x_2^3\}$ to get the corresponding $\overline{M}$ and $\overline{S}$:
    \[\overline{M} = k[x_0^2x_1, \; x_0^2x_2,\; x_0x_1^2, \;x_0x_1x_2, \;x_0x_2^2,\; x_1^2x_2,\; x_1x_2^2],\]
    \[\overline{S} = k[z_3,\; z_4,\; z_5,\; z_6,\; z_7, \; z_8, \;z_9].\]

     We consider the monomial $m = x_0^2x_1$ and apply Lemma \ref{Lem:EEL}. Notice that $M$ inherits its grading from $S$, so elements of $\overline{M}_i$ are the elements of degree $i \cdot d$ in $R$, and $m \in \overline{M}_1$. Thus, if $m$ satisfies the necessary conditions, it will yield nonzero entries in row 1 of the Betti table. Following the notation of the lemma,
    \[D(m) = \{z_3\} \implies |D(m)| = 1\]
    \[A(m) = \{z_3, z_4, z_5, z_6, z_7, z_8\} \implies |A(m)| = 6.\]
    Since $D(m) \subseteq A(m)$, Lemma \ref{Lem:EEL} concludes that $\beta_{1,2}, \; \beta_{2,3},\; \beta_{3,4}, \; \beta_{4.5}, \; \beta_{5,6}$ and $ \beta_{6,7}$ are all nonzero. 

    This example is small enough that a computational software (such as Macaulay2) is able to compute the Betti table, which is given below.
    \begin{center}
        \begin{tabular}{ c | c c c c c c c c c  } 
         & 0 & 1 & 2 & 3 & 4 & 5 & 6 & 7  \\
        \hline
        0 & 1 & - & - & - & - & - & -& - \\ 
        1 & - & 27 & 105 & 189 & 189 & 105 & 27 & -   \\ 
        2 & - & - & - & - & - & - & - & 1  
        \end{tabular} 
    \end{center}
    Notice that syzygy analysis on the monomial $x_0^2x_1$ was sufficient to correctly locate all nonzero entries in the first row.     For the second row, one could take $m$ to be $x_0^2x_1^2x_2^2$ and perform a similar analysis to find that $\beta_{7,9} \neq 0$. Lemma \ref{Lem:EEL} does not guarantee that these are the only nonzero entries in a given row, however Ein-Erman-Lazarsfeld conjecture that all Betti entries outside of the cited range for a well-chosen monomial vanish, as is witnessed in this example. 
    \end{example}

    In lieu of a complete proof of Lemma \ref{Lem:EEL} (which can be found in \cite[Lemma 2.3]{EEL_2016}) we sketch the key components in the setting of Example \ref{ex:EEL}.
    \begin{example}
    \label{ex:EEL2}
        Let $\overline{M}$ and $\overline{S}$ as in the previous example. In order to conclude $\beta_{1,2} \neq 0$, it is sufficient to show the following strand has a nonzero homology element:
         \[\wedge^{0} \overline{S}_1\otimes \overline{M}_{2} \xleftarrow {\;\;\; \phi\;\;} \wedge^1 \overline{S}_1\otimes \overline{M}_1 \xleftarrow {\;\;\; \psi\;\;} \wedge^2 \overline{S}_1 \otimes \overline{M}_{0}.\]
        We claim that $z_1 \otimes x_0^2x_1 \in \wedge^1 \overline{S}_1 \otimes \overline{M}_1$ is in $\ker(\phi) / \operatorname{im}(\psi)$. 
    
        First, we see that 
        \[\phi(z_1 \otimes x_0^2x_1) = 1 \otimes x_0^4x_1^2 = 1 \otimes 0 \in \wedge^0 \overline{S}_1 \otimes \overline{M}_2 \implies z_1 \otimes x_0^2x_1 \in \ker(\phi).\]

        Next, suppose there was some element $(\ell_1 \wedge \ell_2)\otimes 1 \in \wedge^2 \overline{S}_1 \otimes \overline{M}_0$ such that $\psi\left((\ell_1 \wedge \ell_2)\otimes 1\right) = z_1 \otimes x_0^2x_1$. By definition, 
        \[\psi\left((\ell_1 \wedge \ell_2)\otimes 1\right) = \ell_1 \otimes (\ell_2)_{\overline{M}} + \ell_2 \otimes (\ell_1)_{\overline{M}} = z_1 \otimes x_0^2x_1,\] where $(\ell_i)_{\overline{M}}$ denotes the element of $\overline{M}$ that corresponds to $\ell_i$.
    
        However, $z_1$ is the element of $\overline{S}$ that corresponds to $x_0^2x_1$, so the above chain of equalities implies that $\ell_1$ must equal $\ell_2$. Moreover, 
        \[\ell_1 = \ell_2 \implies \ell_1 \wedge \ell_2 = 0 \text{ so that } (\ell_1 \wedge \ell_2) \otimes 1 = 0 \otimes 1 \not\in \wedge^2\overline{S}_1 \otimes \overline{M}_0.\] Thus, there is no such element mapping to $z_1 \otimes x_0^2x_1$, and $z_1 \otimes x_0^2x_1 \not\in \operatorname{im}(\psi)$.

        Combining both statements, we can conclude that $z_1 \otimes x_o^2x_1 \in \operatorname{Tor}^{\overline{S}}_1( \overline{M},k)_{2}$, meaning that $\beta_{1,2} \neq 0.$

        \end{example}

    The argument showing an arbitrary $\beta_{i,j}$ is nonzero follows a similar structure. In general, an element $(z_{a_1} \wedge \ldots \wedge z_{a_i}) \otimes m \in \wedge^{i}\overline{S_1} \otimes M_{j}$ will be in $\ker(\phi)$ (where $\phi$ is the map $\wedge^{i-1} \overline{S}_1 \otimes \overline{M}_{j+1} \longleftarrow \wedge^{i}\overline{S}_1 \otimes \overline{M}_{j}$) if the elements corresponding to $z_{a_1}, \ldots, z_{a_n}$ annihilate $m$ in $\overline{M}$. 
    
    Similarly, we can always arrive at a contradiction that shows $(z_{a_1} \wedge \ldots \wedge z_{a_i}) \otimes m \notin \operatorname{im}(\psi)$ (where $\psi$ is the map $\wedge^{i} \overline{S}_1 \otimes \overline{M}_{j} \longleftarrow \wedge^{i+1}\overline{S}_1 \otimes \overline{M}_{j-1}$) if all the $z_i \in \overline{S}$ that correspond to divisors of $m$ appear in the wedge product. This condition is stronger than what's needed, but it provides a simple condition that can be checked for individual monomials. 

    The two conditions in the above example give a straightforward method for finding nonzero elements of the Betti table: continuing with the notation in Lemma \ref{Lem:EEL}, if we can find an element $z_{a_1} \wedge \ldots \wedge z_{a_{i}} \otimes m \in \wedge^i \overline{S}_1 \otimes M_{j}$ so that $D(m) \subseteq \{z_{a_1}, \ldots, z_{a_{i}} \} \subseteq A(m)$, then $\beta_{i,i+j} \neq 0$. 
    
    For the setting of $\mathbb{P}^n$, Ein-Erman-Lazarsfeld choose an $m$ with $D(m) \subseteq A(m)$ and construct $z_{a_1} \wedge\ldots \wedge z_{a_i} \otimes m$ such that $\{z_{a_1}, \ldots , z_{a_{i}}\} = D(m)$. This yields that $\beta_{|D(m)|, |D(m)| + \deg(m)} \neq 0$. By subsequently adding the elements of $A(m) \setminus D(m)$ to the wedge product, they get successive nonzero Betti entries through $\beta_{|A(m)|, |A(m)| + \deg(m)}$, as in the statement of Lemma \ref{Lem:EEL}.

\subsection{Extending The EEL Method to Weighted Projective Spaces}\hfill
\label{subsec:extendingEEL}

The method outlined above remains largely unchanged in the weighted projective setting; however, the key details involved in executing the basic idea need to be modified. Even in the standard graded setting, it is not a priori clear how many monomials are necessary in the analysis of a given row, or which the best choices are. In \cite{EEL_2016}, the authors choose the lex-leading monomial in each degree to find nonzero Betti entries in the corresponding row. In the nonstandard graded setting, the degrees of the variables are no longer symmetric, and it is unclear which variable to optimize for when selecting a lex ordering. For example, in the setting of $\mathbb{P}(1,1,2)$, should we take the exponent on the degree 2 variable to be as large as possible or as small as possible? The answer is that we need to consider both! This is the first new obstacle of working in the weighted projective setting: row-by-row syzygy analysis is still possible, but may require multiple monomials per row.

\begin{example}
    Consider the 5th Veronese of the ring $R=k[x_0,x_1,y]$ corresponding to $\mathbb{P}(1,1,2)$, with all other notation as defined in the previous examples. Suppose we wish to find nonvanishing syzygies in the second row of the Betti table. 

    First, we optimize for the $x_i$'s and let $m_1 = x_0^4x_1^4y \in \overline{M}_2$. $\overline{S}$ has 10 variables, 8 of which yield divisors of $m_1$ (all but the $z_i$ corresponding to $x_0y^2$ and $x_1y^2)$, so that $|D(m_1)| = 8$. All 10 annihilate $m_1$ in $\overline{M}$, so $D(m_1) \subseteq A(m_1)$, and $|A(m_1)| = 10$. Therefore, the same arguments outlined in Example \ref{ex:EEL2} show $\beta_{8, 10}$, $\beta_{9,11}$, and $\beta_{10,12}$ must all be nonzero. 

    Next, we optimize for $y$ and let $m_2 = x_0^2y^4 \in \overline{M}_2$. The only $z_i \in \overline{S}$ that yields a divisor of $m_2$ is the one corresponding to $x_0y^2$, so $|D(m_2)| = 1$. Since $x_0y^2$ annihilates $m$ in $\overline{M}$, $D(m) \subseteq A(m)$. The generators corresponding to $x_0^2x_1^3$ and $x_0x_1^4$ don't annihilate $m_2$ in $\overline{M}$, but the rest do, so $|A(m_2)| = 8$. This choice of monomial allows us to conclude that the Betti entries $\beta_{1, 3}, \; \beta_{2, 4} ,\ldots , \;\beta_{8, 10}$ are all nonzero. 
    
    This example is again small enough that Macaulay2 can compute the Betti table, given below. 
    \begin{table}[h]
    \begin{center}
        \begin{tabular}{ c | c c c c c c c c c c c c } 
         & 0 & 1 & 2 & 3 & 4 & 5 & 6 & 7 & 8 & 9 & 10  \\
        \hline
        0 & 1 & - & - & - & - & - & -& - & -& - & - \\ 
        1 & - & 43 & 222 & 558 & 840 & 798 & 468 & 147 & 8 & - & -  \\ 
        2 & - & 10 & 88 & 342 & 768 & 1092 & 1008 & 588 & 201 & 20 & 1 \\ 
        3 & - & - & - & - & - & -& -& -& -& 9 & 2
        \end{tabular}
    \caption{\label{tab:BettiTable} The Betti table of the $5$-th Veronese embedding of $\mathbb{P}(1,1,2)$.}
    \end{center}
    \end{table}
    
    For this example, we similarly see that the EEL Method correctly locates all of the nonzero entries in the second row; however, it was necessary to consider both $m_1$ and $m_2$ in order to do so! 
\end{example}

Crucially, the sets of nonzero Betti entries corresponding to $m_1$ and $m_2$ in the above example are distinct, but overlapping. In Section \ref{section:MainResults}, we will similarly construct $m_1$ and $m_2$ for arbitrary $n$, $d$ and row index $q$, find the formulas for their corresponding blocks of nonzero Betti entries, and show they overlap for $d \gg 0$.

\section{Regularity Computations for Veronese Embeddings of General Weighted Projective Spaces}
\label{section:regularity}

An important homological invariant of a module is the Castelnuovo-Mumford regularity. The following definition is phrased as in \cite[Definition 1.3]{BrownErmanPositivity_2023}, but the initial definition dates back to \cite[~\S 5]{Benson2004}. 
\begin{definition}
\label{def:reg}
    Let $M$ be a module over some (possibly nonstandard) graded polynomial ring $S$, and let $\mathfrak{m}$ denote the maximal ideal of $S$ given by its variables. We say that $M$ is \textit{(weighted) $r$-regular} if $H^i_{\mathfrak{m}}(M)_j = 0$ for all $i \geq 0$ and $j > r-i$. The \textit{(weighted) Castelnuovo-Mumford regularity of M}, $\reg(M)$, is the smallest integer $r$ for which $M$ is $r$-regular. 
\end{definition}

For the purposes of this article, we will refer to the weighted Castelnuovo-Mumford regularity simply as the regularity of a module. 

In this section, we wish to explicitly compute the regularity of the $d$th Veronese of any weighted projective space. We need to be a bit careful when defining the degree $d$ Veronese embedding of a weighted projective space; as was noted in the introduction, $\mathcal{O}(d)$ may not be very ample for an arbitrary $\mathbb{P}(a_0, \ldots, a_n)$, and so the Veronese may not be given by a complete linear series of a line bundle in this setting. Instead, we utilize the traditional algebraic definition to define the degree $d$ Veronese of a weighted projective space.

\begin{definition}
\label{def:weightedVer}
    Let $\mathbb{P}(a_0, \ldots, a_n)$ be some weighted projective space. Let $R= k[x_0, \ldots, x_n]$, where $\deg(x_i) = a_i$. We define the degree $d$ Veronese of $\mathbb{P}(a_0, \ldots, a_n)$ as the degree $d$ Veronese subring of $R$,
    \[R^{(d)} := \bigoplus_{e\geq 0} R_{ed}.\]
\end{definition}

Therefore, for $S$ a graded polynomial ring with a minimal surjection $S \to R^{(d)}$, we can view $R^{(d)}$ as an $S$-module, where the action is given by the map. Let $M=R^{(d)}$, viewed as an $S$-module. We wish to compute $\reg(M)$. Since we are interested in the asymptotic syzygies of $M$, it suffices to classify $\reg(M)$ for $d \gg 0$.

\begin{lemma}
\label{lem:regularity}
    For $M$ as defined above, if $d \geq \sum_{i=0}^n a_i$, then $\reg(M) = n$. 
\end{lemma}

\begin{proof}
    By Definition \ref{def:reg}, $M$ is $r$-regular if $H^i_\mathfrak{m}(M)_j = 0$ for all $i \geq 0$ and $j > r-i$. Since $M$ is Cohen-Macaulay of dimension $n+1$, its only nonzero local cohomology module is $H^{n+1}_\mathfrak{m}(M)$. Thus, for our setting, $M$ is $r$-regular if $H^{n+1}_\mathfrak{m}(M)_j = 0$ for all $j > r-n-1$.

    Note that there is an equivalence between local cohomology modules and sheaf cohomology modules given by $H^i_m(M)_j \cong H^{i-1}\left(\Proj(S), \widetilde{M}(j)\right)$ for all $i > 1$, so that showing $H^{n+1}_\mathfrak{m}(M)_j = 0$ is equivalent to showing $H^n\left(\Proj(S), \widetilde{M}(j) \right) = 0$. Moreover, it suffices to show that $r = n$ is the smallest value of for which $H^n\left(\Proj(S), \widetilde{M}(j)\right) = 0$ for all $j > r-n-1$, or, equivalently, that $H^n(\Proj(S), \widetilde{M}(j))=0$ for all $j > -1$ and $H^n(\Proj(S), \widetilde{M}(-1)) \neq 0$.

    First, observe that 
    \[H^n\left(\Proj(S), \widetilde{M}(j)\right) = H^n\left(\Proj(S), i_*\mathcal{O}_{\Proj(R)}(j)\right) = H^n\left(\Proj(R), \mathcal{O}_{\Proj(R)}(d\cdot j)\right).\] 
    
    Recall that 
    \[H^n\left(\Proj(R), \mathcal{O}_{\Proj(R)}(e)\right) \neq 0 \iff e \leq - \sum_{i=0}^n a_i,\] where the $a_i$ are the degrees of the variables in $R$. If $d \geq \sum_{i=0}^n a_i$ and $j \leq -1$ (equivalently, $r < n$), then 
    \[(r-n) \leq -1  \implies d\cdot j \leq -d \leq -\sum_{i=0}^n a_i,\]
    and the above condition yields that $H^n\left(\Proj(R), \mathcal{O}_{\Proj(R)}(d\cdot -1)\right) = H^n\left( \Proj(S), \widetilde{M}(-1)\right)\neq 0$. However, when $j > 0$ (equivalently, $r \geq n)$, we have 
    \[d\cdot j \geq 0 \geq -\sum_{j=0}^n a_j \implies H^n\left( \Proj(S), \widetilde{M}(j)\right) = 0.\]
    Thus, when $d\geq \sum_{i=0}^n a_i$, $n$ is the smallest value of $r$ for which $M$ is $r$-regular, meaning that $\reg(M) = n$, as claimed. 
\end{proof}

In the standard graded setting, the Castelnuovo-Mumford regularity corresponds to the index of the bottom-most nonzero row of $\beta^S(M)$, so that $\reg(M) = r$ means that $r$ is the largest integer with $\beta_{i, i + r} \neq 0$ for any $i$. When we move to the nonstandard graded setting, we need to make a slight adjustment: $\reg(M) = r$ means that the bottom-most nonzero row of the Betti table is $r + \sigma(S)$ where $\sigma(S)$ is the Symonds' constant, as defined in \cite[Proposition 1.2]{Symonds2011}. In particular, for $S = k[z_0, \ldots , z_{N}]$ with $\deg(z_i) = b_i$, $\sigma(S) = \sum_{i=0}^{N} (b_i - 1)$. Note that when $S$ happens to be standard graded, $\sigma(S) = 0$, so this formula states the bottom-most row of the Betti table is $r$, aligning with the standard graded results. 

For our purposes, the take-away is that the index bottom-most row of the Betti table depends on the degrees of the variables of $S$ in addition to $\reg(M)$. 

\begin{example}
    Consider the weighted projective space $\mathbb{P}(1,1,2)$ and let $d =5$. This choice of $d$ satisfies the statement of Lemma \ref{lem:regularity} in this setting as $\sum_{i=0}^n a_i = 4$. 
    
    Let $R = k[x_0, x_1, y]$ be the corresponding nonstandard graded polynomial ring. Then, 
    \[R^{(5)} = k[x_0^5, \; x_0^{4}x_1, \ldots, x_1^5, \; x_0^{3}y,\; x_0^{1}y^2, \; x_1^{3}y, \; x_1y^2,\; x_0^2x_1y, \; x_0x_1^2y,\; y^5].\]
    Note that, for this setup, the only $k$-algebra generator that is not in $R_5$ is $y^5 \in R_{10}$. 

    Let $S = k[z_0, \ldots, z_{12}]$ be the polynomial ring corresponding to the image of the Veronese embedding (so that each $z_i$ corresponds to a generator of $R^{(5)}$ in the above order). Since we rescale the grading by $5$ in $S$, $\deg(z_i) = 1$ for $0 \leq i \leq 11$, and $\deg(z_{12}) = 2$. Thus, \cite{Symonds2011} shows that  
    \[\sigma(S) = \sum_{i = 0}^{12} \left(\deg(z_i) - 1\right) = 1, \] 
    and, since $n=2$ for this example, $\beta^S(M)$ extends through row $\reg(M)+ \sigma(S) = 2+1 = 3$. The Betti table for this example is given as Table \ref{tab:BettiTable}, which verifies that we have correctly found the number of rows. 
\end{example}

In the above example, we chose $\mathbb{P}(1,1,2)$ with $d = 5$ in order to simplify the notation. However, if we had taken a different projective space of the form $\mathbb{P}(1^n,2)$ and any $d$ odd, the resulting $M$ would still have exactly one degree $2$ generator, $y^d$, which we prove below. We state the below lemma in terms of $R$ to make the graded pieces clear, but the statement on $M$ is the same with the grading rescaled. 

\begin{lemma}
\label{lem:dOddgens}
    Let $R=k[x_0, \ldots, x_{n-1},y]$ be the nonstandard graded polynomial ring with $\deg(x_i) = 1$ and $\deg(y) = 2$. If $d \gg 0$ is odd, then $R^{(d)}$ has exactly one $k$-algebra generator of degree $2d$, $y^d$, and the remaining generators have degree $d$.
\end{lemma}

\begin{proof}
    As in Definition \ref{def:weightedVer}, we let $R^{(d)} = \oplus_{e= 0}^\infty R_{ed}$. Monomials in $R_{ed}$ are of the form $\mathbf{x^a}y^b$, where $\mathbf{x^a} = x_0^{a_0}\cdots x_{n-1}^{a_{n-1}}$ for some exponent vector $\mathbf{a} = (a_0, \ldots, a_{n-1})$, with $|\mathbf{a}| + b = id$. All monomials in $R_d$ will be $k$-algebra generators of $R^{(d)}$ by construction. We aim to show the only remaining generator is $y^d \in R_{2d}$. 
    
     Let $\mathbf{x^a}y^b$ be some monomial in $R_{2d}$ with $\mathbf{a} \neq \overline{0}$. We claim this monomial can be rewritten as the product of two monomials in $R_d$. First, recall that $\mathbf{a}$ is a vector, so for any $s,t \geq 0$ with $s+t = |\mathbf{a}|$, we can find vectors $\mathbf{a_s}$ and $\mathbf{a_t}$ such that $|\mathbf{a_s}| = s$, $|\mathbf{a_t}| = t$, and $\mathbf{a_s} + \mathbf{a_t} = \mathbf{a}$. Since $x^ay^b \in R_{2d}$, $|a| + 2b$ is even, so $|a|$ is always even, and we split in cases based on the parity of $b$. 
     
     If $b$ is even, we can use the above fact to find $\mathbf{a_s}$ and $\mathbf{a_t}$, each with magnitude $\frac{|\mathbf{a}|}{2}$, so that $\mathbf{x^a}y^b = \mathbf{x}^{\mathbf{a_s}}y^{\frac{b}{2}} \cdot \mathbf{x}^{\mathbf{a_t}} y^{\frac{b}{2}}$. In particular, $\deg(\mathbf{x}^{\mathbf{a_s}}y^{\frac{b}{2}}) = \deg(\mathbf{x}^{\mathbf{a_t}}y^{\frac{b}{2}}) = \frac{|\mathbf{a}|}{2} + 2\cdot \frac{b}{2} = d$, since $|\mathbf{a}| + 2b = 2d.$
    Thus, $\mathbf{x}^{\mathbf{a_s}}y^{\frac{b}{2}}$ and $\mathbf{x}^{\mathbf{a_t}}y^{\frac{b}{2}}$ are monomials in $R_d$, and $\mathbf{x^a}y^b$ can be written as product of existing generators. 
    
    If $b$ is odd, we can similarly find $\mathbf{a_s}$ with $|\mathbf{a_s}| = \frac{|\mathbf{a}|}{2} -1$ and $\mathbf{a_t}$ with $|\mathbf{a_t}| = \frac{|\mathbf{a}|}{2}+1$, so that $\mathbf{a} = \mathbf{a_s} + \mathbf{a_t}$. Then, $\mathbf{x}^{\mathbf{a_s}}y^{\frac{b+1}{2}}$ and $\mathbf{x}^{\mathbf{a_t}}y^\frac{b-1}{2}$ are both monomials in $R_d$ that multiply to $\mathbf{x^a}y^b$, so it is decomposable into existing generators. 

    The remaining elements of $R_{2d}$ are those with $\mathbf{a} = \overline{0}$. The only such element is $y^d$. Since $d$ is odd, and $\deg(y) = 2$, there is no element of the form $y^b$ in $R_d$, and we cannot write $y^d$ as the product of existing generators. Thus, $y^d$ must be a $k$-algebra generator of $R^{(d)}$.

    Lastly, we claim that every monomial of $R_{ed}$ for $e \geq 3$ can be generated by $R_d \cup \{y^d\}.$ We do so inductively. Suppose all $R_{e'd}$ with $e' < e$ can be generated by the given set and let $\mathbf{x^a}y^b \in R_{ed}$. Either $|\mathbf{a}|$ or $b$ must be greater than or equal to $d$ in order for $|\mathbf{a}| + 2b \geq 3d$. If $b \geq d$, we have $\mathbf{x^a}y^b = \mathbf{x^a}y^{b-d} \cdot y^d$. Since $\mathbf{x^a}y^{b-d} \in R_{(e-2)d}$ and $y^d \in \langle y^d \rangle$, this element can be generated by the given set. If $|\mathbf{a}| \geq d$, we can find $\mathbf{a_s}$ and $\mathbf{a_t}$ with magnitudes $d$ and $|\mathbf{a}|-d$, respectively, such that $\mathbf{a} = \mathbf{a_s} + \mathbf{a_t}$. Then, $\mathbf{x^a}y^b = \mathbf{x}^{\mathbf{a_s}} \cdot \mathbf{x}^{\mathbf{a_t}}y^b$. Since $\mathbf{x}^{\mathbf{a_s}} \in R_d$ and $\mathbf{x}^{\mathbf{a_t}}y^b \in R_{(e-1)d}$, this element can also be generated by the given set. 

    Thus, the $k$-algebra generators of $R^{(d)}$ are the monomials in $R_d$ (all of which have degree $d$) and $y^d$ (which has degree $2d$). 
\end{proof}

Since we rescale the grading on $S$ and $M$, the above lemma shows that $M$ has precisely one generator of degree 2, and the remaining generators are degree 1, as desired. Moreover, for an arbitrary $\mathbb{P}(1^n,2)$, the Symonds' constant, $\sigma(S)$, is always 1 (for $d$ odd), so when $d \gg 0$ and odd, $\beta^S(M)$ will extend through row $n + 1$. This result will be necessary for our main results, so we label it as a corollary. 

\begin{corollary}
\label{cor:lastRow}
    For any $n\geq 1$, and any $d \geq n+2$ and odd, the bottom-most row of the Betti table of the $d$th Veronese embedding of $\mathbb{P}(1^n, 2)$ has index $n+1$. 
\end{corollary}

The below lemma highlights what happens for $\mathbb{P}(1^n,2)$ when $d$ is even.

\begin{lemma}
\label{lem:dEvenGens}
    Let $R = k[x_0, \ldots, x_{n-1}, y]$ with $\deg(x_i) = 1$ and $\deg(y) =2$. When $d$ is even, all of the $k$-algebra generators of $R^{(d)}$ have degree $d$. In particular, the bottom-most row of corresponding Betti table has index $n$. 
\end{lemma}
\begin{proof}
    The monomials of degree $d$ in $R$ are generators of $R^{(d)}$ by construction. In this case, a pure power of $y$ is included in this set as $y^{\frac{d}{2}} \in R_d$. 

    We claim any monomial in $R_{ed}$ with $e \geq 2$ can be written as a product of monomials in $R_d$. By induction, suppose this true for all $R_{e'd}$ with $e' < e$ and let $\mathbf{x^a}y^b \in R_{ed}$, so that $|\mathbf{a}| + 2b = ed$. Either $|\mathbf{a}| \geq d$ or $b \geq \frac{d}{2}$ in order for $|\mathbf{a}| + 2b \geq 2d$. 

    If $b \geq \frac{d}{2}$, we have $\mathbf{x^a}y^b = \mathbf{x^a}y^{b-\frac{d}{2}} \cdot y^{\frac{d}{2}}$. Since $\mathbf{x^a}y^{b-\frac{d}{2}} \in R_{(e-2)d}$ and $y^{\frac{d}{2}}\in R_d$, this element can be generated by the given set. 

    If $|\mathbf{a}| \geq d$, we can find $\mathbf{a_s}$ and $\mathbf{a_t}$ with magnitudes $d$ and $|\mathbf{a}|-d$, respectively, such that $\mathbf{a} = \mathbf{a_s} + \mathbf{a_t}$. Then, $\mathbf{x^a}y^b = \mathbf{x}^{\mathbf{a_s}} \cdot \mathbf{x}^{\mathbf{a_t}}y^b$. Since $\mathbf{x}^{\mathbf{a_s}} \in R_d$ and $\mathbf{x}^{\mathbf{a_t}}y^b \in R_{(e-1)d}$, this element can also be generated by the given set. 

    Moreover, all of the $k$-algebra generators of $R^{(d)}$ have degree $d$, so the ring $S=k[z_0, \ldots, z_N]$ where each $z_i$ corresponds to a generator is standard graded (after rescaling). Thus, $\sigma(S) = 0$ when $d$ is even, and the corresponding Betti table extends through row $n$. 
\end{proof}

The previous two results highlight a key phenomenon, and a second obstacle of working with weighted projective spaces: the Betti tables corresponding to different Veronese embeddings of the same weighted projective space can have a different number of rows, unlike the standard graded setting. It is expected that the number of rows will depend on the degrees of the variables in the nonstandard graded setting; however, the Veronese degree also impacts the number of rows in the Betti table.

In this article, we will focus on weighted projective spaces of the form $\mathbb{P}(1^n, 2)$. To begin motivating the decision to specialize to this case, we replicate the above analysis for $\mathbb{P}(1,1,3)$. 

\begin{example}
    \label{ex:P(1,1,3)Reg}
    Let $R = k[x_0,x_1,y]$ be the nonstandard graded polynomial ring where $\deg(x_0) = \deg(x_1) = 1$ and $\deg(y) = 3$. Consider the 5th Veronese, 
    \[M = R^{(5)} = k[x_0^5, \; x_0^4x_1, \ldots, x_1^5, \; x_0^2y, \; x_0x_1y, \; x_1^2y, \; \bm{x_0y^3}, \; \bm{ x_1y^3}, \; \bm{y^5}].\] 
    The $k$-algebra generators written in bold are those supplying variables with degree $> 1$ in the corresponding $S$. In particular, $S$ will have two variables of degree 2 and one variable of degree 3, so that $\sigma(S) = 4$. Therefore, for this particular value of $d$, we expect the Betti table to extend through the sixth row. The Betti table for this setting, shown in Example \ref{ex:futureDir}, confirms this computation.
\end{example}

Observe that, even for a relatively small example over $\mathbb{P}(1,1,3)$, the Betti table becomes much larger. However, the real complication is that, unlike $\mathbb{P}(1^n,2)$, there are more than two possibilities for the value of $\sigma(S)$. For example, when $d = 7$, $\sigma(S) = 2$, and when $d = 9$, $\sigma(S) = 0$, each of which result in Betti tables with a different number of rows. Thus, the analysis we are interested in becomes more delicate as we would not only need to consider the cases arising from varying $n$, but also the different possible $\sigma(S)$ values for each $n$.

\section{Notation}
\label{section:notation}
For the remainder of this article, we narrow our scope to $\mathbb{P}(1^n,2)$ and adopt the following notation.

Let $R := k[x_0, \ldots, x_{n-1},y]$ denote the nonstandard $\mathbb{Z}_{\geq 0}$-graded polynomial ring corresponding to the weighted projective space $\mathbb{P}(1^n, 2)$, so that $\deg(x_i) = 1$ and $\deg(y) = 2$. In particular, we will consider the setting where $n > 1$.

We construct $R^{(d)}$ as described in Definition \ref{def:weightedVer} and let $S$ be a graded polynomial ring with minimal surjection $S \to R^{(d)}$. In particular, we let $S:=k[z_0, \ldots, z_N]$ be the polynomial ring where the $z_i$ correspond to the $k$-algebra generators of $R^{(d)}$. The induced map on $\Proj$ is the $d$th Veronese embedding in this setting: $\varphi_d\colon \mathbb{P}(1^n,2) \to  \Proj(S)$. Note that the grading in $S$ is rescaled by $d$, so that $\deg(z_i)$ is the degree of the corresponding monomial in $R$ divided by $d$. We order the $z_i$'s so that $z_0, \ldots, z_n$ correspond to the pure power generators. Concretely, $z_0, \ldots, z_{n-1}$ correspond to $x_0^d, \ldots x_{n-1}^d$, so $\deg(z_0) = \ldots = \deg(z_{n-1}) = 1$. However, since $\deg(y) =2$, the degree of $z_n$ will depend on the parity of $d$. If $d$ is even, $z_n$ corresponds to $y^{\frac{d}{2}}$, and $\deg(z_n) = 1$. The more interesting case occurs when $d$ is odd, as $z_n$ corresponds to $y^d$, and $\deg(z_n) = 2$, so the ring $S$ is also nonstandard graded.

\begin{remark}
    Proposition 3.2 of \cite{beLinSyzOCurvesIWProjSp} confirms that the Veronese map is an embedding in this setting. For example, one could choose $\ell = 2$ and apply the proposition. 
\end{remark}

Let $M := R^{(d)}$, viewed as an $S$-module, where the $z_i$ act according to their corresponding $k$-algebra generator of $M$. For example, if $m \in M$, then $z_0 \cdot m = x_0^d \cdot m \in M$. As in Section \ref{section:eel}, it will be useful to perform an Artinian reduction on $M$, so we state it explicitly for $d$ even and odd. 

\subsection{Artinian Reduction}
Let $d$ be even and consider $\overline{M} = M/\langle x_0^d, \ldots, x_{n-1}^d, y^\frac{d}{2} \rangle$, viewed as an $\overline{S}$-module, where $\overline{S} = S/\langle z_0, \ldots, z_n \rangle$. Since $\{x_0^d, \ldots, x_{n-1}^d,\; y^\frac{d}{2}\}$, is a regular sequence on $M$, $\overline{M}$ is an Artinian reduction of $M$, and the Betti table of $\overline{M}$ over $\overline{S}$ equals the Betti table of $M$ over $S$. Recall that, when $d$ is even, $S$ is standard graded, so the benefit of performing the Artinian reduction is that $\overline{M}$ is a finite $\overline{S}$-module. When $d$ is odd, $S$ is nonstandard graded, so the Artinian reduction is particularly advantageous as it converts $M$ to a a finite module over a standard graded ring. 

In this case, let $d$ be odd and consider $\overline{M} = M/\langle x_0^d, \ldots, x_{n-1}^d, y^d \rangle$, viewed as an $\overline{S}$-module, where $\overline{S} = S/\langle z_0, \ldots, z_n \rangle$. $\overline{M}$ is an Artinian reduction of $M$ as $\{x_0^d, \ldots, x_{n-1}^d,\; y^d\}$ is again a regular sequence on $M$. Thus, in either case, it is sufficient to show the entries of $\beta^{\overline{S}}(\overline{M})$ are nonzero at the desired locations.

\begin{corollary}
    $\overline{S}$ as defined above is standard graded. 
\end{corollary}
\begin{proof}
    When $d$ is even, $S$ is already standard graded (as is shown in Lemma \ref{lem:dEvenGens}), so $\overline{S}$ must be standard graded as well. 
    
    The more meaningful result is for $d$ odd, but this is an immediate consequence of Lemma \ref{lem:dOddgens}, which states that $R^{(d)}$ has precisely one $k$-algebra generator of degree $2d$, $y^d$, and the rest are degree $d$. Equivalently, $S$ has precisely one variable of degree 2 (which we denote by $z_n$). The Artinian reduction removes this variable, so the resulting $\overline{S}$ is standard graded, containing only variables of degree 1. 
\end{proof}

\begin{remark}
    The above corollary fails for most other weighted projective spaces! In general, if $R$ is any nonstandard graded polynomial ring, $R^{(d)}$ will have $k$-algebra generators in degree $\neq d$ that are not pure powers of the variables, which causes the corresponding $\overline{S}$ to be nonstandard graded. We discuss this complication further in Section \ref{section:futureDir}. 
\end{remark}

For specific monomials $m \in \overline{M}$, we will let $D(m) \subseteq \{z_{n+1}, \ldots z_{N}\}$ be the list of $z_i$ such that the corresponding monomial in $\overline{M}$ divides $m$. Let $A(m)$ be the list such that the corresponding monomial annihilates $m$ in $\overline{M}$.  

The following example details all of the notation.
\label{ex:P(1,1,2),d=3 Notation}
\begin{example}
    Let $R = k[x_0, x_1, y]$, with $\deg(x_0)=\deg(x_1) = 1$, and $\deg(y) = 2$, and consider $d = 3$. 
    \[M = R^{(d)} = k[x_0^3, \; x_1^3,\; y^3,\; x_0^2x_1,\; x_0x_1^2,\; x_0y,\; x_1y],\] and $S = k[z_0, \ldots, z_6]$ where each $z_i$ corresponds to a $k$-algebra generator of $M$, in the order shown. Then, $\overline{M} = M/\langle x_0^3,\; x_1^3,\; y^3 \rangle$ and $\overline{S} = S/\langle z_0, z_1, z_2 \rangle \cong k[z_3, z_4, z_5, z_6]$. 

    For $m = x_1y,$ $D(m) = \{z_6\}$ and $A(m) = \{z_4\}$. As an aside, notice that this $m$ does not satisfy $D(m) \subseteq A(m)$, so it is not a valid choice of monomial for the EEL Method. 
\end{example}

\section{Main Results}
\label{section:MainResults}
Assume the notation defined in Section \ref{section:notation}, and let $d \gg 0$ be an integer. In this section, we will prove the theorems from the introduction. 

We will index the rows in the Betti table by $q$ and state results in terms of $q$, $d$, and $n$. As a result of the nonstandard grading, the range for $q$ will depend on the parity of $d$. If $d$ is even, we will consider $q \in [1, n]$ (due to Lemma \ref{lem:dEvenGens}). If $d$ is odd, we will consider $q \in [1, n+1]$ (Corollary \ref{cor:lastRow}). For $q$ in the appropriate range, we let $F_q(d)$ and $B_q(d)$ denote the left and right-most column indices (respectively) in which we guarantee a nonzero entry in row $q$. These constructions will be made more explicit in their respective subsections below. 

As the $d$ odd case is more nuanced, we will describe the necessary methods in detail for $d$ odd first. In Sections \ref{subsec:FrontResults} and \ref{subsec:BackResults}, we compute $F_q(d)$ and $B_q(d)$ for $d$ odd by applying the EEL Method to specific monomials. As discussed in Section \ref{subsec:extendingEEL}, the monomials used for $F_q(d)$ and $B_q(d)$ will often be different, so we must also confirm that the blocks of nonzero entries corresponding to each monomial overlap. We show the necessary overlap for $d$ odd in Section \ref{subsec:Overlap}. Then, in Section \ref{subsec:evenVerDeg} we show analogous $F_q(d)$, $B_q(d)$, and overlap results for $d$ even. The theorems from the introduction are a combination of the results from each of these sections, and are shown in Section \ref{subsec:MainThrm}. 

We first provide a lemma that gives an explicit formula for computing the Hilbert functions of nonstandard graded polynomial rings with some degree 1 variables and one degree 2 variable, which will be very useful in the computations of $F_q(d)$ and $B_q(d)$. The following lemma is stated in terms of a ring $R_{i,1}$ as that's how we will apply it in the later subsections. For now, one can see the utility of such a lemma because it will allow us to explicitly compute $N$ (see Remark \ref{rmk:NevenNodd}). 

\begin{lemma}
    \label{Lem:GenHilb}
    Let $R_{i,1}= k[x_0, \ldots, x_{i-1}, y]$ denote the nonstandard graded polynomial ring with $i$ variables of degree 1 and one variable of degree 2. Then, for arbitrary $s$, 
    \[\operatorname{Hilb}(s, R_{i,1}) = \sum_{b=0}^{\floor{\frac{s}{2}}} {s-2b + i - 1 \choose i - 1}.\]
\end{lemma}
\begin{proof}
    Recall that $\operatorname{Hilb}(s, R_{i,1}) = \dim_k\left((R_{i,1})_s\right)$. We will count the total number of degree $s$ monomials in $R_{i,1}$ as these form a basis for $(R_{i,1})_s$ as a vector space over $k$. A monomial of degree $s$ in this ring will be of the form $\mathbf{x^a}y^b$ (where $\mathbf{a}=(a_0, \ldots, a_{i-1})$ is some exponent vector and $\mathbf{x^a}$ denotes $x_0^{a_0}\cdots x_{i-1}^{a_{i-1}}$) with $|\mathbf{a}| + 2b = s$. We can break this computation into pieces by fixing the exponent on $y$ and counting the number of monomials of the correct degree under that constraint. 
    
    Since we require $|\mathbf{a}| + 2b = s$, we consider $b$ from 0 to $\floor{\frac{s}{2}}$. The number of degree $s$ monomials in $R_{i,1}$ containing $y^b$ for a fixed $b$ is the same as the number of degree $s - 2b$ monomials in $R_{i,1} / \langle y \rangle$. Importantly, $R_{i,1}/\langle y \rangle$ is a standard graded polynomial ring (with $i$ variables), so the usual combinatorial formulas apply. In particular, for a fixed $b$, we have $s - 2b + i-1 \choose i-1$ monomials. Summing over all choices of $b$, we get the desired formula. 
\end{proof}

\begin{remark}
    \label{rmk:NevenNodd}
    The above lemma gives a direct way for computing $N = \dim \Proj(S)$. The precise formula will depend on the parity of $d$, so we state both cases here. If we consider $R_{n,1} = R$, $\operatorname{Hilb}(d,R)$ measures the degree $1$ variables of $S$. When $d$ is odd, Lemma \ref{lem:dOddgens} shows that this is all the variables of $S$ but one, so $\operatorname{Hilb}(d,R) = \dim \Proj(S) = N$ in this case. Therefore,
    \[N = \operatorname{Hilb}(d, R) = \sum_{b=0}^{\frac{d-1}{2}} {d - 2b + n -1 \choose n - 1} \text{ when } d \text{ is odd}.\]

    When $d$ is even, Lemma \ref{lem:dEvenGens} shows $S$ is generated in degree 1, so that $\operatorname{Hilb}(d,R) = N+1$, and  
    \[N = \operatorname{Hilb}(d, R) - 1 = \sum_{b=0}^{\frac{d}{2}} {d - 2b + n -1 \choose n - 1} -1 = \sum_{b=0}^{\frac{d}{2}-1} {d - 2b + n -1 \choose n - 1} \text{ when } d  \text{ is even}.\]  
\end{remark}

\vspace{10pt}
The following definition and lemma will allow us to compute the order of magnitude of $N$ with respect to $d$ (and, more generally, the order of magnitude of $\operatorname{Hilb}(s,R_{i,1})$ with respect to $s$). 

\begin{definition}
We say that a function $F: \mathbb{Z}\to \mathbb{Z}$ \textit{agrees with a period $2$ polynomial of degree $\leq b$ for $s\gg 0$} if there exists a pair of polynomials $(p_{\text{even}}, p_{\text{odd}})$, each of degree $\leq b$, such that 
\[F(s) = \begin{cases}
    p_{\text{even}}(s) & \text{if } s \text{ is even} \\
    p_{\text{odd}}(s) & \text{if } s \text{ is odd}
\end{cases} \]
for all $s \gg 0$. We say that the periodic polynomial has degree equal to $b$ if both $p_{\text{even}}$ and $p_{\text{odd}}$ have degree $b$.
\end{definition}

\begin{lemma}
\label{lem:asympN}
Let $R_{i,1}$ be a polynomial ring as defined in Lemma \ref{Lem:GenHilb}. 
    \begin{enumerate}
        \item\label{cor:asympN:poly} There exists a period $2$ polynomial $(p_{\text{even}}, p_{\text{odd}})$ with degree equal to $i$ which agrees with $\Hilb(s,R_{i,1})$ for $s \gg 0$. Moreover, $p_{\text{even}}$ and $p_{\text{odd}}$ have the same leading coefficient.
        \item\label{cor:asympN:hypersurface} If $f\in R_{i,1}$ is any positive degree, homogeneous element and $Q=R_{i,1}/f$, then there exists a period $2$ polynomial $(p'_{even}, p'_{odd})$ of degree $\leq i-1$ which agrees with $\Hilb(s,Q)$ for $s \gg 0$.
    \end{enumerate}
\end{lemma}
\begin{proof}
(\ref{cor:asympN:poly}) For this lemma, let $R = R_{i,1}$, and consider $R^{(2)}$ as a module over the polynomial ring $S$, constructed as is outlined in Section \ref{section:notation} for this choice of $R$ with $d = 2$. Let $R_{\text{even}} = \bigoplus_{j \geq 0, \text{ even}} R_j = R^{(2)}$ and $ R_{\text{odd}} = \bigoplus_{j \geq 0, \text{ odd}} R_j$, both viewed as $S$-modules (the action of $S$ on $R_{\text{odd}}$ is the same as the action of $S$ on the Veronese). Note that, for this choice of $d$, $S$ is a standard graded polynomial ring (by Lemma \ref{lem:dEvenGens}).

In the standard graded setting, the Hilbert function of a $S$-module, $M$, eventually agrees with a polynomial of the form 
\[\left(\frac{e(M)}{(\dim(M)-1)!}t^{\dim(M)-1} + \text{ lower order terms }\right),\] 
where $e(M)$ is the multiplicity of the module. (The proof of this fact can be found in \cite[~\S 4]{bhCMRings}. Theorem 4.1.3 states the Hilbert function eventually agrees with a polynomial of degree $\dim(M)-1$, and Definition 4.1.5 gives a formula for the leading coefficient, as $e_0$ = $e(M)$ when $\dim(M) > 0$). 

For both $M = R_{\text{even}}$ and $M = R_{\text{odd}}$, $\{x_0^2, \ldots x_{i-1}^2, y\}$ forms a regular sequence on $M$ in $R$, so that $\text{depth}(m_R, M) \geq i+1$. However, 
\[i+1=\dim(R) \geq \dim(M) \geq \text{depth}(m_R, M) = i+1,\] so we must have that $\dim(R_{\text{even}}) = \dim(R_{\text{odd}}) = i+1.$ Thus, the Hilbert functions of $R_{\text{even}}$ and $R_{\text{odd}}$ both eventually agree with polynomials of degree $(i+1)-1 = i$. In order to show these polynomials have the same leading coefficient, it suffices to show that $e(R_{\text{even}}) = e(R_{\text{odd}})$. 

Multiplicity is invariant under an Artinian reduction, so that $e(R_{\text{even}}) = \text{length}\left( R_{\text{even}} / \langle x_0^2, \ldots, x_{i-1}^2, y \rangle \right)$. Since we take the quotient by a monomial ideal, this can be found by counting the number of nonzero monomials in $R_{\text{even}} / \langle x_0^2, \ldots, x_{i-1}^2, y \rangle$ (and similarly for $e(R_{\text{odd}})$). 

In particular, $e(R_{\text{even}})$ is the number of square-free monomials in with even degree the $x_j$'s, which is known to be $2^{i-1}$. The total number of square free monomials in the $x_j$'s is known to be $2^i$, so $e(R_{\text{odd}})$, the number of square-free, odd-degree monomials in the $x_j$'s, is given by $2^i - 2^{i-1} = 2^{i-1}$. Thus, $e(R_{\text{even}}) = e(R_{\text{odd}})$, and the statement follows. 
 
(\ref{cor:asympN:hypersurface}) For any $q \in R$, and $Q = R/f$ we have the short exact sequence 
\[0 \longrightarrow R(-\deg(f)) \longrightarrow R \longrightarrow Q \longrightarrow 0, \] so that 
\[\Hilb(d,Q) = \Hilb(d, R) - \Hilb(d-\deg(f), R). \]
By (\ref{cor:asympN:poly}), we know that, for $s \gg 0$, the Hilbert function over $R$ agrees with a period 2 polynomial $(p_{\text{even}}, p_{\text{odd}})$, where $p_\text{even}$ and $p_{\text{odd}}$ are both of the form $ct^i + \text{lower order terms}$ (where $c = \frac{2^{i-1}}{i!}$). 

We construct the necessary $(p'_{\text{even}}, p'_{\text{odd}})$ in cases. If $\deg(f)$ is even, we define
\[p'_{\text{even}}(t):= p_{\text{even}}(t) - p_{\text{even}}(t-\deg(f)) \text{ and } p'_{\text{odd}}(t):= p_{\text{odd}}(t) - p_{\text{odd}}(t-\deg(f)). \]
If $\deg(f)$ is odd, we can similarly construct 
\[p'_{\text{even}}(t):= p_{\text{even}}(t) - p_{\text{odd}}(t-\deg(f)) \text{ and }p'_{\text{odd}}(t):= p_{\text{odd}}(t) - p_{\text{even}}(t-\deg(f)).\] 

The statement on Hilbert functions combined with the results in (\ref{cor:asympN:poly}) shows that the Hilbert function of $Q$ will eventually agree with the period 2 polynomial $(p'_{\text{even}}, p'_{\text{odd}})$. It remains to show these polynomials have the desired degree. 

In either case, 
\[p'_{\text{even / odd}}(t) = ct^i + \text{ lower terms } - c(t-\deg(f))^i- \text{ lower terms}.\]
The degree $i$ term will cancel, so that the degrees of $p'_{\text{even}}$ and $p'_{\text{odd}}$ must both be less than to $i$, as claimed. 
\end{proof}

\begin{remark}
\label{rmk:asypN}
    Part (\ref{cor:asympN:poly}) of the above Lemma shows that, for $s \gg 0$, $\operatorname{Hilb}(s, R_{i,1})$ agrees with a period 2 polynomial, where each piece has the form $\frac{2^{i-1}}{i!} t^i +$ (lower order terms). Moreover, this means that $\operatorname{Hilb}(s,R_{i,1})$ is $\mathcal{O}(s^i)$. In particular, $N$, as defined in Remark \ref{rmk:NevenNodd}, is asymptotically $d^n$ (regardless of the parity of $d$). 
\end{remark}

 \subsection{Front of Betti Table for Odd Veronese Degree}\hfill
 \label{subsec:FrontResults}

Recall from Section \ref{section:eel} that a monomial $m \in \overline{M}_q$ corresponds to a nonzero block of entries in row $q$ of the Betti table if $D(m) \subseteq A(m)$. In this case, the column index of the left-most entry of this block is given by $|D(m)|$. In order to describe the front of the Betti table, we will find a monomial $m$ that is expected to have a relatively small number of divisors, then compute $|D(m)|$ explicitly. 

Since $y$ is the only variable of degree 2, the $y$-heaviest monomial is a good choice for such an $m$. We denote $F_q(d)$ to be $|D(m)|$ where $m$ is the lex-most monomial in $\overline{M}_q$ with respect to the order $y > x_0 > \ldots > x_{n-1}$. However, for $q = 1$, this monomial is $x_0y^{\frac{d-1}{2}}$, which doesn't satisfy $D(m) \subseteq A(m)$, so we need to choose a different monomial for the first row. We let $F_1(d) = |D(x_0^{d-1}x_1)|$. For either case, one can think of $F_q(d)$ as the column index of the left-most Betti entry that we guarantee to be nonzero in row $q$.

We find $F_1(d)$ separately first, and then address rows $2 \leq q \leq n+1$ in Lemma \ref{lem:Fqd}. 

\begin{lemma}
    \label{lem:F1d}
    For $\mathbb{P}(1^n,2)$, regardless of the parity of $d$, $F_1(d) = 1$.  
\end{lemma}
\begin{proof}
    Let $m = x_0^{d-1}x_1 \in \overline{M}_1$. This $m$ corresponds to some $z_i \in \overline{S}_1$, so $D(m) = \{z_i\}$. Observe that $D(m) \subseteq A(m)$ as, for $d > 1$, 
    \[z_i \cdot m = x_0^{d-1}x_1 \cdot x_0^{d-1} x_1 = x_0^{2d-2} x_1^2 = 0 \in \overline{M}.\] We define $F_1(d) = |D(m)|$ for this choice of $m$, so $F_1(d) = 1$. 
\end{proof}

The above lemma shows that, for any $d > 1$, the first row of the Betti table is guaranteed to have nonzero entries beginning in column 1. The following lemma describes the left-most guaranteed nonzero entry in the remaining rows for $d \gg 0$ odd. 

\begin{lemma}
\label{lem:Fqd}
    For arbitrary $2 \leq q \leq n+1$ and $d \gg 0$ odd, \[F_q(d) = \operatorname{Hilb}(d,R_{q-1,1}) - \operatorname{Hilb}(d-q-1, R_{q-1,1}) - (q-2),\] where $R_{q-1,1} = k[x_0, \ldots, x_{q-2},\; y] \subseteq R.$
    Equivalently, 
    \[F_q(d) = \sum_{b=0}^{\frac{d-1}{2}} {d - 2b + q - 2 \choose q - 2} - \sum_{b=0}^{\floor*{\frac{d-q-1}{2}}} {d-2b-3 \choose q-2} - q+2. \]
\end{lemma}

\begin{proof}
    Consider $m = x_0^{d-1} \cdots x_{q-3}^{d-1} x_{q-2}^q y^{d-1} \in R_{q-1,1}$. First, observe that $m \in \overline{M}_q$ as 
    \[ \deg(m) = \frac{(q-2)(d-1) + q + 2(d-1)}{d} = \frac{(qd - 2d - q + 2) + q + (2d - 2)}{d} = \frac{qd}{d} = q.\] 

    Any $z_i \in D(m)$ must correspond to a monomial containing either an $x_j$ for $j \in [0, q-3]$ or a $y$ (as there are no pure power $x_{q-2}$ terms in $\overline{M}$). Any such $z_i$ also corresponds to an annihilator of $m$ in $\overline{M}$, so $D(m) \subseteq A(m)$.
    
    In this setting, $D(m)$ is in bijection with the set of degree $d$ generators of the ring $R_{q-1,1}/I$ for $I  = \langle x_0^d, \ldots ,x_{q-3}^d, x_{q-2}^{q+1}, y^d \rangle$. Therefore, $|D(m)| = \operatorname{Hilb}(d, R_{q-1,1}/I) $.

    We can compute this Hilbert function explicitly. First, we resolve $R_{q-1,1}/I$ over $R_{q-1,1}$: 
    \[\mathcal{F}_\bullet: 0 \longleftarrow R_{q-1,1}/I \longleftarrow R_{q-1,1} \longleftarrow R_{q-1,1}(-q-1) \oplus R_{q-1,1}(-d)^{q-2} \oplus R_{q-1,1}(-2d)\longleftarrow \cdots.\]

    Hilbert functions satisfy an alternating sum on resolutions, so $\operatorname{Hilb}(d, R_{q-1,1}/I)$ is equal to
    \[\operatorname{Hilb}(d, R_{q-1,1}) - \Big( \operatorname{Hilb}(d, R_{q-1,1}(-q-1)) + (q-2) \operatorname{Hilb}(d, R_{q-1,1}(-d)) + \operatorname{Hilb}(d, R_{q-1,1}(-2d)) \Big)+ \cdots.\]

    To simplify this calculation, notice that $F_j$ for $j \geq 2$ will be comprised of free modules $R_{q-1,1}(-r)$ with $r > d$ (since twists appearing in $F_{i+1}$ will be sums of twists in $F_i$). Furthermore, 
    \[\operatorname{Hilb}\left(d,R_{q-1,1}(-r)\right) = \operatorname{Hilb}\left(-r+d, R_{q-1,1}\right) = 0 \text{ when } r > d,\] so none of the $F_j$ (for j $\geq 2$) contribute to the alternating sum. Similarly, $\operatorname{Hilb}\left(d, R_{q-1,1}(-2d)\right) = 0$, so that
    \[|D(m)| = \operatorname{Hilb}(d, R_{q-1,1}/I) = \operatorname{Hilb}(d, R_{q-1,1}) -\operatorname{Hilb}(d, R_{q-1,1}(-q-1)) - (q-2)\operatorname{Hilb}(d, R(-d)).\]

    Moreover, observe that
    \[\operatorname{Hilb}\left(d, R_{q-1,1}(-d)\right) = \dim_k \left((R_{q-1,1}(-d))_d\right) = \dim_k\left((R_{q-1,1})_{-d+d} \right)= \dim_k\left((R_{q-1,1})_0 \right)= 1.\] 
    Similarly, 
    \[\operatorname{Hilb}\left(d,R_{q-1,1}(-q-1)\right) = \dim_k \left((R_{q-1,1})_{d-q-1}\right) = \operatorname{Hilb}(d-q-1, R_{q-1,1}), \text{ so that }\]
    \[|D(m)| = \operatorname{Hilb}(d, R_{q-1,1}/I) = \operatorname{Hilb}(d, R_{q-1,1}) -\operatorname{Hilb}(d-q-1, R_{q-1,1}) - (q-2).\]

    Since we let $F_q(d) = |D(m)|$ for this choice of $m$, we have shown the first statement in the lemma. 

   It remains to show that the given closed form equals $F_q(d)$, which we will do by finding explicit formulas for each Hilbert function in the above statement via Lemma \ref{Lem:GenHilb}. Since $R_{q-1,1}$ has $q-1$ variables, we replace $i$ with $q-1$ in the formula and apply it to each degree we wish to compute:
   \[\operatorname{Hilb}(d, R_{q-1,1}) = \sum_{b = 0}^{\frac{d-1}{2}} {d - 2b + q - 2 \choose q-2} \text{ and }\]
    \[ \operatorname{Hilb}(d-q-1, R_{q-1,1}) = \sum_{b = 0}^{\floor*{\frac{d-q-1}{2}}} {d-q-1 - 2b + q-2\choose q - 2} = \sum_{b = 0}^{\floor*{\frac{d-q-1}{2}}} {d - 2b - 3 \choose q - 2}.\]

    Finally, we can reconstruct $F_q(d)$ using the corresponding sums of binomial coefficients so that 
    \begin{align*}
        F_q(d) & = \operatorname{Hilb}(d, R_{q-1,1}) - \operatorname{Hilb}(d-q-1, R_{q-1,1}) - (q-2) \\& = \sum_{b=0}^{\frac{d-1}{2}} {d - 2b + q - 2 \choose q - 2} - \sum_{b=0}^{\floor*{\frac{d-q-1}{2}}} {d-2b-3 \choose q-2} - q+2. 
    \end{align*}
    
    \end{proof}

\begin{example}
    \label{ex:Fq5}
    Let's consider setting $\mathbb{P}(1,1,2)$ with $d = 5$. This example is small enough that Macaulay2 is able to compute the Betti table (which is given as Table \ref{tab:BettiTable}), and we can compare with the bounds provided by the above lemma. 

    When $q = 2$, we get 
    \[F_2(5) = \sum_{b=0}^2 {5 -2b + 2 - 2 \choose 0} - \sum_{b=0}^1 {5 - 2b -3 \choose 2 - 2} - 2 +2\]
    \[ = \left( {5 \choose 0} + {3 \choose 0} + {1 \choose 0}\right) -\left( {2 \choose 0} + {0 \choose 0}\right) = 3 - 2 = 1,\]
    so the second row of the Betti table must also begin in the first column.

    For $q = 3$, we have
    \[F_3(5) = \sum_{b=0}^2 {5 - 2b + 1 \choose 1} - \sum_{b=0}^0 {5 - 2b - 3 \choose 1} - 3+2 \]
    \[= \left( {6 \choose 1} + {4 \choose 1} + {2 \choose 1} \right) -  {2 \choose 1} - 1 = 9,\]
    so the third row of the Betti table is guaranteed to have a block of nonzero entries beginning in column 9. 

    Notice that these are precisely the positions where each row of the Betti table begins! Although we haven't shown $F_q(d)$ is a sharp bound, we see that it is for this example, which serves as our first piece of evidence in the direction of Conjecture \ref{conjecture}. 
\end{example}

\begin{remark}
    The provided formula for $F_q(d)$ does not depend on $n$ in any way. This means that, for example, the third row of the Betti table corresponding to the 5th Veronese embedding of $\mathbb{P}(1,1,1,2)$ (or more generally, any $\mathbb{P}(1^n,2)$ with $n > 1$) will also be guaranteed to have a block of nonzero entries beginning in column $9$. 
\end{remark}

\subsection{Back of Betti Table for Odd Veronese Degree}\hfill
\label{subsec:BackResults}

Since $\overline{M}$ is a finite $\overline{S}$-module,  Hilbert's Syzygy Theorem states that the nonzero entries of $\beta^{\overline{S}}(\overline{M})$ (equivalently $\beta^S(M)$) must be contained within columns $0$ and $\dim_k(\overline{S})$. Recall that $S$ has $N + 1$ variables, so $\overline{S}$ has $N + 1- (n+1) = N-n$ variables, by construction. Therefore, the maximum column index of a nonzero Betti entry is $N-n$ for all $q$, providing an upper bound on the $B_q(d)$ values. 

In order to compute $B_q(d)$ explicitly, we similarly wish to select a monomial $m$ of the correct degree that yields nonzero Betti entries. In order to analyze the back of the Betti table, we want this $m$ to have a large number of annihilators, as the right-most nonzero Betti entry guaranteed by $m$ has column index $|A(m)|$. 

The monomial of $\overline{M}_q$ that is the lex-most with respect to $x_0 > \ldots > x_{n-1} > y$ will have the most variables raised to the power of $d-1$, up to symmetry. Such a monomial will be annihilated by the most $x_i$'s, and is thus a good choice for our $m$. We let $B_q(d) = |A(m)|$ for this particular $m$.

Instead of counting the number of annihilators of $m$, we will find a formula for the number of non-annihilators, which we call $|NA(m)|$, and use this to calculate $B_q(d)$. In particular,  
\[B_q(d) = |A(m)| = N-n- |NA(m)|.\]
We consider the cases where $q \leq n-1$, $q = n$, and $q = n+1$ separately. 

\begin{lemma}
    \label{lem:genBqd}
    Let $q \in [1, n-1]$ and $d \gg 0$. Then,  
    \[B_q(d) = N  - \operatorname{Hilb}(d, R_{n-q,1}) + \operatorname{Hilb}(q, R_{n-q,1}) -q-1, \] where $R_{n-q,1} = k[x_q, \; \ldots, x_{n-1}, \; y] \subseteq R$.
    Equivalently, 
    \[B_q(d) = N - \sum_{b = 0}^{\frac{d-1}{2}} {d-2b + n-q-1 \choose n-q-1} + \sum_{b=0}^{\floor*{\frac{q}{2}}} {-2b+n-1 \choose n-q-1} -q-1. \]
\end{lemma}
\begin{proof}
     Consider the monomial $m = x_0^{d-1}\cdots x_{q-1}^{d-1}x_q^{q}$. Since we assume $q \leq n-1$ and $d \gg 0$ (so that $d \geq q$), $m \in \overline{M}$. Every monomial of $\overline{M}$ that divides $m$ must contain at least one $x_i$ with $i < q$, so every divisor also annihilates $m$. This, combined with the fact that 
    \[ \deg(m) = \frac{q(d-1) + q}{d} = q,\] 
    implies that $m$ yields a block of nonzero entries in the $q$-th row of the Betti table.

    The degree $d$ elements of $R_{n-q,1}/ I$, where $I =\langle x_q^{d-q}, x_{q+1}^{d}, \ldots x_{n-1}^d, y^d \rangle$ are in bijection with the $z_i \in \overline{S}$ that do not annihilate $m$ in $\overline{M}$. Therefore, $|NA(m)| = \operatorname{Hilb}(d,R_{n-q,1}/I)$. As in the proof of Lemma \ref{lem:Fqd}, we will do so by resolving $R_{n-q,1}/I$ and taking an alternating sum of Hilbert functions. 

    $R_{n-q,1}/I$ has minimal free resolution: 
    \[ \mathcal{F}_\bullet : 0 \longleftarrow R_{n-q,1}/I \longleftarrow R_{n-q,1} \longleftarrow R_{n-q,1}(-d+q) \oplus R_{n-q,1}(-d)^{n-q-1} \oplus R_{n-q,1}(-2d) \longleftarrow F_2 \longleftarrow\cdots.\]

    As before, $\operatorname{Hilb}\left(d, R_{n-q,1}(-r)\right) = 0$ when $r > d$, so $\operatorname{Hilb}(d, R_{n-q,1}(-2d)) = 0$. Similarly, all free modules appearing in $F_{i+1}$ will have twists that are sums of twists appearing in $F_{i}$, meaning that all free modules in $F_2$ (and beyond) will not contribute to the Hilbert function computation in degree $d$. Thus, 
    \[ \operatorname{Hilb}(d, R_{n-q,1}/I) = \operatorname{Hilb}(d, R_{n-q,1}) - \operatorname{Hilb}\left(d, R_{n-q,1}(-d+q)\right) - (n-q-1)\cdot \operatorname{Hilb}\left(d, R_{n-q,1}(-d)\right).\]

    Recall that $\operatorname{Hilb}\left(d, R_{n-q,1} (-d)\right) = \operatorname{Hilb}(0,R_{n-q,1}) = 1$. It remains to compute $\operatorname{Hilb}(d, R_{n-q,1})$ and $\operatorname{Hilb}\left(d, R_{n-q,1}(-d+q)\right) = \operatorname{Hilb}(q, R_{n-q,1})$. We apply Lemma \ref{Lem:GenHilb} to each, which yields
    \[ \operatorname{Hilb}(d, R_{n-q,1}) = \sum_{b = 0}^{\frac{d-1}{2}} {d-2b + n-q-1 \choose n-q-1} \text{ and }\] 
    \[ \operatorname{Hilb}(q, R_{n-q,1}) = \sum_{b=0}^{\floor*{\frac{q}{2}}} {q-2b + n-q-1 \choose n-q-1} = \sum_{b=0}^{\floor*{\frac{q}{2}}} {-2b+n-1 \choose n-q-1}.\] 
    
    Putting the pieces together, we have
    \begin{align*}
        B_q(d) &=  N - n -\big(\operatorname{Hilb}(d, R_{n-q,1}) - \operatorname{Hilb}(q, R_{n-q,1}) - (n-q-1)\big) \\
        &= N - n - \Big(\sum_{b = 0}^{\frac{d-1}{2}} {d-2b + n-q-1 \choose n-q-1} - \sum_{b=0}^{\floor*{\frac{q}{2}}} {-2b+n-1 \choose n-q-1} - (n-q-1) \Big) \\
        &= N - \sum_{b = 0}^{\frac{d-1}{2}} {d-2b + n-q-1 \choose n-q-1} + \sum_{b=0}^{\floor*{\frac{q}{2}}} {-2b+n-1 \choose n-q-1} -q-1.
    \end{align*}
    \end{proof}

\begin{lemma}
\label{lem:secondLastBqd}
    For $d \gg 0$ and odd, the second to last row of the Betti table, where $q = n$, satisfies 
    \[B_n(d) = \begin{cases} 
      N-n & \text{if }  n \text{ is even} \\
      N-n-1 & \text{if }  n \text{ is odd}.
   \end{cases} \]
\end{lemma}

\begin{proof}
    The largest monomial with respect to lex on $x_0 > \ldots > x_{n-1} > y$ in $\overline{M}_{n}$ depends on the parity of $n$, so we consider the cases where $n$ is even and odd separately.

    First, we consider the case where $n$ is even. Let $m$ be the monomial $x_0^{d-1} \cdots x_{n-1}^{d-1} y^{\frac{n}{2}}$. We assume $d \gg 0$, so $d > \frac{n}{2}$, and this monomial is nonzero in $\overline{M}$. Observe that every $z_i \in \overline{S}$ corresponds to a monomial containing at least one of the $x_i$'s, so that every $z_i$ annihilates $m$ in $\overline{M}$. In particular, $|NA(m)| = 0$ and $D(m)$ must be contained in $A(m)$. Since
    \[ \deg(m) = \frac{n(d-1) + 2\left(\frac{n}{2}\right)}{d} = \frac{nd - n + n}{d} = n,\] $m$ yields a nonzero block in row $q = n$, and $B_n(q) = N - n - |NA(m)| = N-n$ when $n$ is even. 

    If $n$ is odd, we consider the monomial $m = x_0^{d-1} \cdots x_{n-2}^{d-1} x_{n-1}^{d-2} y^{\frac{n+1}{2}}$. Unlike the previous case, there is a $k$-algebra generator of $\overline{M}$ that doesn't annihilate $m$: $x_{n-1}y^{\frac{d-1}{2}}$. Let $z_j$ be the variable corresponding to this monomial. Observe that all other $z_i \in \overline{S}$ correspond to a monomial containing either an $x_k$ with $k \neq n-1$ or $x_{n-1}^b$ with $b \geq 2$, so that every $z_i$ with $i \neq j$ must annihilate $m$.
    
    Note that, since $d \gg 0$, $x_{n-1}y^{\frac{d-1}{2}}$ does not divide $m$, and $D(m) \subseteq A(m)$. Lastly, since $\deg(m) = n$, we can conclude that $B_n(d) = N - n - |NA(m)| = N-n-1$ when $n$ is odd. 
\end{proof}

\begin{lemma}
\label{lem:lastBqd}
    For $d \gg 0$ and odd, the last row of the Betti table, where $q= n+1$, extends to the farthest possible column, so that $B_{n+1}(d) = N-n$.
\end{lemma}

\begin{proof}
    As in the previous lemma, we consider the cases where $n$ is even and odd separately. 

    First, suppose $n$ is odd. Consider the monomial $m = x_0^{d-1}\cdots x_{n-1}^{d-1}y^{\frac{d+n}{2}}.$ Note that $\frac{d+n}{2}$ is an integer as we assume $d$ is also odd. Additionally, we require $d \gg 0$, so we may assume $\frac{d+n}{2} < d$. Observe that every $z_i \in \overline{S}$ annihilates $m$ in $\overline{M}$ as each $z_i$ corresponds to a monomial containing at least one of the $x_i$. Therefore, $D(m) \subseteq A(m)$ and $|NA(m)| = 0$. Since
    \[ \deg(m) = \frac{n(d-1) + d+n}{d} = \frac{nd - n + d + n}{d} = \frac{(n+1)d}{d} = n+1,\] $B_{n+1}(q) = N-n$ when $n$ is odd. 

    Next, assume $n$ is even, and let $m = x_0^{d-1} \cdots x_{n-2}^{d-1}x_{n-1}^{d-2}y^{\frac{d+n+1}{2}}.$ Unlike the previous case, a bit of work is required to show that every $z_i \in \overline{S}$ annihilates $m$ in $\overline{M}$. Suppose, for contradiction, that there were some $z_j$ that did not annihilate $m$. This $z_j$ would need to correspond to a monomial of the form $x_{n-1}y^b$ where $b < \frac{d-n-1}{2}$ and $1 + 2b = d$. However, 
    \[b < \frac{d-n-1}{2} \implies 1 + 2b < 1 + d-n-1 = d-n,\] so such an element cannot exist. Thus, all the $z_i \in \overline{S}$ must annihilate $m$ (so that $|NA(m)| = 0$), and, again, $D(m) \subseteq A(m)$ trivially. In this case, we have
    \[\deg(m) = \frac{(n-1)(d-1) + (d-2) + (d+n+1)}{d} = \frac{nd + d}{d} = n+1,\] so $B_{n+1}(q) = N-n$ when $n$ is even as well. 
\end{proof}
    
    \begin{remark}
        Lemma \ref{lem:lastBqd} provides a sharp characterization for where the last row of the Betti table terminates; Hilbert's Syzygy Theorem yields that a row cannot extend beyond column $N-n$.
    \end{remark} 
    
    The following example shows how each of the lemmas in this section can be applied. 

    \begin{example}
        Let's return to our running example of $\mathbb{P}(1,1,2)$ with $d = 5$. 

        Lemma \ref{lem:lastBqd} states that row $n+1 = 3$ of the Betti table must extend as far right as possible (i.e. through column $N-n$). In this setting, $N = 12$ (since $S$ has $13$ total variables), so we guarantee that row $3$ of the Betti table extends through column $10$. Since, $n+1$ is odd, Lemma \ref{lem:secondLastBqd} yields that row $2$ must extend through column $10$ as well. Lastly, we can use Lemma \ref{lem:genBqd} to compute $B_1(5)$:
        \[B_1(5) = 12 - \sum_{b = 0}^{2} {5-2b + 0 \choose 0} + \sum_{b=0}^{0} {-2b+ 1 \choose 0} -2 = 12 - (1+1+1) + 1 -2 = 8,\]
        so that row $1$ of the Betti table must extend at least through column $8$. 

        The Betti table for this example (Table \ref{tab:BettiTable}), shows that each row terminates at exactly these locations. 
    \end{example}

    If we momentarily take for granted that the nonzero blocks corresponding to the monomials used to compute $F_q(d)$ and $B_q(d)$ overlap (which is shown in the next section), combining the above example with Example \ref{ex:Fq5} gives a complete picture of the Betti table. Notably, the bounds found for the left and right-most nonzero Betti entries of each row are sharp, providing further evidence towards Conjecture \ref{conjecture}. 

    \begin{remark}
            The methods described in this article are particularly useful in cases where computing the Betti table explicitly is too time-consuming. This happens often; for example, the $7$th Veronese embedding of $\mathbb{P}(1,1,1,2)$ already exceeds the computational capabilities of Macaulay2. 
    \end{remark}

\subsection{Overlap for Odd Veronese Degree}\hfill
\label{subsec:Overlap}

Sections \ref{subsec:FrontResults} and \ref{subsec:BackResults} provide formulas for the left and right-most Betti entries that we guarantee to be nonzero. However, since different monomials were used to determine each of these positions, there is one final step: ensuring that all Betti entries between the given $F_q(d)$ and $B_q(d)$ are nonzero as well. In this section, we will again use a row-by-row analysis to show the blocks of nonzero Betti entries corresponding to the two monomials overlap. 

\begin{lemma}
\label{lem:oddOverlap}
    Assume $d \gg 0$ and odd. For any $1 \leq q \leq n+1$, $\beta_{i, i+q} \neq 0$ for $F_q(d) \leq i \leq B_q(d)$. 
\end{lemma}

\begin{proof}
    We aim to show that, for each $q$, $|D(m_2)| \leq |A(m_1)|$ where $m_1$ and $m_2$ are the monomials used to compute $F_q(d)$ and $B_q(d)$, respectively. We consider the rows $q = 1$, $2 \leq q \leq n-1$, $q = n$, and $q=n+1$ separately, as each case has a distinct $m_1$, $m_2$ pair. 

    \textbf{Row 1:} 
    When $q =1$, $m_1 = m_2 = x_0^{d-1}x_1$ (from Lemmas \ref{lem:F1d} and \ref{lem:genBqd}, respectively), so there is no overlap to check in this row.

    \textbf{Rows} $\bm{2 \leq q \leq n-1\colon}$
    The corresponding monomials are $m_1 = x_0^{d-1}\cdots x_{q-3}^{d-1}x_{q-2}^qy^{d-1}$ and $m_2 = x_0^{d-1}\cdots x_{q-2}^{d-1}x_q^q$ (from Lemmas \ref{lem:Fqd} and \ref{lem:genBqd}). We first compute $|D(m_2)|$. We can apply the methods from Lemma \ref{lem:Fqd} and realize the divisors of $m_2$ as the degree $d$ elements in $R_{q+1,0}/ I:= k[x_0, \ldots, x_q] /\langle x_0^d, \ldots, x_{q-1}^d, x_q^{q+1}\rangle$. As before, we have $|D(m_2)| = \operatorname{Hilb}(d, R_{q+1,0}/I)$, which can be computed via an alternating sum on the resolution of $R_{q+1,0}/I$, so that
    \[\operatorname{Hilb}(d, R_{q+1,0}/I) = \operatorname{Hilb}(d, R_{q+1,0}) - \operatorname{Hilb}\left(d, R_{q+1,0}(-q-1)\right) - q. \]
    $R_{q+1,0}$ is a standard graded polynomial ring with $q+1$ variables, so the usual combinatorial formulas apply, and 
    \[|D(m_2)| = {d+q \choose q} - {d-1 \choose q} - q.\]
    
    Next, we wish to compute $|A(m_1)|$. Applying the methods of Lemma \ref{lem:genBqd}, we can realize the set of non-annihilators in $\overline{M}$ as $R_{n-q+2,0}/I':=  k[x_{q-2}, \ldots, x_{n-1}] /\langle x_{q-2}^{d-q},x_{q-1}^d, \ldots, x_{n-1}^d \rangle$. In particular, $|NA(m_1)| = \operatorname{Hilb}(d, R_{n-q+2,0}/I')$. Applying the same techniques, we have 
    \[|NA(m_1)| = \operatorname{Hilb}(d, R_{n-q+2,0}) - \operatorname{Hilb}\left(d, R_{n-q+2,0}(-d+q)\right) - (n-q+1).\]
    
    $R_{n-q+2,0}$ is a standard graded polynomial with $n-q+2$ variables, and so the usual combinatorial formulas yield 
    \[|A(m_1)| = N - n - |NA(m_1)| = N-n- {d + n - q + 1 \choose n - q + 1} + {n+1 \choose n-q+1} + (n-q+1).\]
    
    Lastly, we wish to show that $|D(m_2)| \leq |A(m_1)|$. We assume $d \gg 0$, so it suffices to show this is true asymptotically. Since $q$, $n$, ${n+1 \choose n-q+1}$, and $(n-q+1)$ do not depend on $d$, these terms are dominated by the others in their respective formulas. 
    
    Recall that $N$ is asymptotically $d^n$ (Remark \ref{rmk:asypN}) so that, for $d\gg 0$,
    \[N = ad^n + (\text{lower order terms}),\] where $a$ is some constant.
    It is also known that so that ${n \choose k}$ is asymptotically $n^k$, so, for $d \gg 0$, 
    \[{d+n-q+1 \choose n-q+1} = a'd^{n-q+1} +  (\text{lower order terms}), \] where $a'$ is also some constant. Therefore, in our setting, 
    \[|A(m_1)| = (ad^n + \text{ lower terms})- \left(a'd^{n-q+1}+\text{ lower terms}\right).\]
    Since $q>1$, $n-q+1<n$ and $|A(m_1)| = ad^{n} +$ (lower order terms). 
    
    Similarly, 
    \[ |D(m_2)| =  (bd^q + \text{ lower terms}) - (b'd^q + \text{ lower terms})\]
    for some constants $b$ and $b'$.  In particular, $|D(m_2)| = b''d^q +$ (lower order terms), where $b'' = b-b'$. 
    
     Since we assume $q \leq n-1$, we have
    \[|D(m_2)| =b''d^q + (\text{lower terms}) \leq ad^{n} + (\text{lower terms}) = |A(m_1)|\] 
    for $d \gg 0$, as desired. 

    \textbf{Row} $\bm{n}\colon$
    The monomial used to compute $F_n(d)$ is the same as the previous case (with $q = n$), so we have $m_1 = x_0^{d-1} \cdots x_{n-3}^{d-1}x_{n-2}^n y^{d-1}$. The monomial $m_2$ depends on the parity of $n$, so we consider two cases. 
    
    If $n$ is even, we have $m_2 = x_0^{d-1} \cdots x_{n-1}^{d-1}y^{\frac{n}{2}}$ (from Lemma \ref{lem:secondLastBqd}). In order to show $|D(m_2)| \leq |A(m_1)|$, we first determine how many elements are in $D(m_2)$ that are not in $A(m_1)$. The only $k$-algebra generators of $\overline{M}$ that do not annihilate $m_1$ have the form $x_{n-2}^ax_{n-1}^b$ with $a<d-n$ and $b < d$ with $a+b=d$. There are $d-n-1$ total such monomials ($x_{n-2}^1x_{n-1}^{d-1}, \ldots, x_{n-2}^{d-n-1}x_{n-1}^{n+1}$), and they all divide $m_2$. Thus, there are exactly $d-n-1$ monomials that are in $D(m_2)$ but not $A(m_1)$, so that $|D(m_2)| - (d-n-1) \leq |A(m_1)|.$
    
    If there are more than $d-n-1$ elements that are in $A(m_1)$ but not $D(m_2)$, we will have the desired inequality. Generators of $\overline{M}$ that do not divide $m_2$ must have the form $\mathbf{x^a}y^b$ where $b > \frac{n}{2}$ and $|\mathbf{a}| + 2b = d$ (for $\mathbf{a} = (a_0, \ldots, a_{n-1})$ some exponent vector with each $a_i < d$). 
    
    If we consider the case where $b = \frac{n}{2}+1$, any possible $\mathbf{a}$ with $|\mathbf{a}| =d-n-2$ will yield such a monomial. Using the standard graded combinatorial formulas, there are ${d-n-2+n-1 \choose n-1} = {d-3 \choose n-1}$ choices for $\mathbf{a}$ satisfying the necessary requirements. When $b = \frac{n}{2}+2$, we can similarly determine that there are an additional ${d-5 \choose n-1}$ such elements.  Since $n \geq 2$, there are at least ${d-3 \choose 1} + {d - 5 \choose 1} = 2d-8$ elements in $A(m_1)$ that are not in $D(m_2)$, so \[|D(m_2)| -(d-n-1) + 2d-8 \leq |A(m_1)|.\] 
   
    For $d \gg 0$, $2d - 8 \geq d-n-1,$ which yields that $|D(m_2)| \leq |A(m_1)|$, as desired. 
    
    If $n$ is odd,  we have $m_2 = x_0^{d-1} \ldots x_{n-2}^{d-1} x_{n-1}^{d-2}y^{\frac{n+1}{2}}$. An identical argument to the even case yields that there are $d-n-2$ elements in $D(m_2)$ that are not in $A(m_1)$, while there are at least ${d-4 \choose n-1} + {d-6 \choose n-1}$ elements in $A(m_1)$ that are not in $D(m_2)$. Again, we can conclude that $|D(m_2)| \leq |A(m_1)|$.

    \textbf{Row} $\bm{n+1}\colon$
    Lemma \ref{lem:lastBqd} yields that row $q = n+1$ of the Betti table extends through column $N-n$. The monomial for $m_1$ is the same as what's used in the above rows, so we get the same formula for $|NA(m_1)|$, 
    \[|NA(m_1)| = { d + n - q + 1 \choose n - q + 1} - {n + 1 \choose n - q + 1} - (n-q+1).\] However, for this row, we have $q = n+1$, so we can simplify the formula to 
    \[= {d + n - (n+1) + 1 \choose n - (n+1) + 1} - {n+1 \choose n - (n+1) + 1} - (n-(n+1) + 1) = {d \choose 0} - {n+1 \choose 0} - 0 = 0.\]
    
    Therefore, the block of nonzero entries corresponding to this $m_1$ must extend through column $N-n-0 = N-n$. Thus, the entire row can be spanned by this sole monomial, and there's no need to check overlap!

    The above row-by-row analysis shows that, for row $q \in [1, n+1]$, every Betti entry between column $F_q(d)$ and $B_q(d)$ must also be nonzero. Concretely, $\beta_{i, i+q} \neq 0$ for $F_q(d) \leq i \leq B_q(d)$. 
\end{proof}

\begin{remark}
    Observe that the row $n+1$ argument in the above lemma gives an alternate proof of Lemma \ref{lem:lastBqd}. In the original proof, we found the largest monomial with respect to lex on $x_0 > \ldots> x_{n-1} > y$ and showed it was annihilated by $z_i \in \overline{S}$; however, the above argument shows we could have instead considered $x_0^{d-1} \cdots x_{n-2}^{d-1}x_{n-1}^{n+1}y^{d-1}$.
\end{remark}

\subsection{Results for Even Veronese Degrees}\hfill
\label{subsec:evenVerDeg}

In this section, we will state the equivalent results about $F_q(d)$ and $B_q(d)$ when $d$ is even. This case introduces less novelty, and the techniques are the same as in the corresponding proofs for $d$ odd, so we omit some of the repeated details. 

Lemma \ref{lem:F1d} shows that $F_1(d) = 1$. The remaining $F_q(d)$ are given by the following lemma.

\begin{lemma}
\label{lem:evenFqd}
    For $d$ even and $2 \leq q \leq n$, 
    \[F_q(d) = \operatorname{Hilb}(d, R_{q,1}) - \operatorname{Hilb}(d-q-2, R_{q,1}) -q.\]
    Equivalently, 
    \[F_q(d) = \sum_{b=0}^{\frac{d}{2}} { d-2b+q-1 \choose q-1} - \sum_{b=0}^{\floor{\frac{d-q-2}{2}}} {d-2b-3 \choose q-1} -q.\]
\end{lemma}
\begin{proof}
    The monomial used to compute $F_q(d)$ for $2 \leq q \leq n$ is $m = x_0^{d-1}\cdots x_{q-2}^{d-1}x_{q-1}^{q+1}y^{\frac{d}{2}-1}$. $D(m)$ is in bijection with the degree $d$ monomials in $R_{q,1}/I$ for $R_{q,1} = k[x_0, \ldots, x_{q-1}, y]$ and $I = \langle x_0^d, \ldots, x_{q-2}^d, x_{q-1}^{q+2}, y^{\frac{d}{2}} \rangle$. The minimal free resolution is
    \[0 \longleftarrow R_{q,1}/I \longleftarrow R_{q,1} \longleftarrow R_{q,1}(-q-2) \oplus R_{q-1}(-d)^q \longleftarrow \cdots.\]
    No other terms in the free resolution contribute to the alternating sum on Hilbert functions in degree $d$, so 
    \[|D(m)| = F_q(d) = \operatorname{Hilb}(d, R_{q,1}) - \operatorname{Hilb}(d-q-2, R_{q,1}) - q.\]
    Applying Lemma \ref{Lem:GenHilb} to each Hilbert function in the above statement yields that this is equivalent to the desired closed form. 
\end{proof}

\begin{lemma}
\label{lem:evenBqd1}
    For $d$ even, and $1 \leq q \leq n-1$, 
    \[B_q(d) = N -\operatorname{Hilb}(d, R_{n-q, 1}) + \operatorname{Hilb}(q, R_{n-q, 1}) -q.\]
    Equivalently, 
    \[B_q(d) = N - \sum_{b=0}^{\frac{d}{2}} { d-2b+n-q-1 \choose n-q-1} + \sum_{b=0}^{\floor{\frac{q}{2}}} {-2b+n-1 \choose n-q-1} -q.\]
\end{lemma}
\begin{proof}
    We consider the monomial $m = x_0^{d-1}\cdots x_{q-1}^{d-1}x_q^q$, and again aim to count the number of non-annihilators. $NA(m)$ is in bijection with the degree $d$ monomials of $R_{n-q, 1} / I$, where $R_{n-q, 1} = k[x_q, \ldots ,x_{n-1}, y]$ and $I = \langle x_q^{d-q}, x_{q+1}^d, \ldots , x_{n-1}^d, y^{\frac{d}{2}}\rangle$. Taking the minimal free resolution then the alternating sum of Hilbert functions yields 
    \[|NA(m)| = \operatorname{Hilb}(d, R_{n-q,1}) - \operatorname{Hilb}(q, R_{n-q, 1}) - (n-q).\]
    Therefore, 
    \[B_q(d) = N-n - \operatorname{Hilb}(d, R_{n-q,1}) + \operatorname{Hilb}(q, R_{n-q, 1}) + (n-q).\]
    The given closed form arises from applying Lemma \ref{Lem:GenHilb} to each of the above Hilbert functions. 
\end{proof}

\begin{lemma}
\label{lem:EvenBqd2}
    For $d$ even, $B_n(d) = N-n$.
\end{lemma}
\begin{proof}
    Our choice of monomial $m$ depends on the parity of $n$, so we consider the case where $n$ is even and odd separately. 

    If $n$ is even, $m = x_0^{d-1} \cdots x_{n-1}^{d-1}y^{\frac{n}{2}}$. Every $z_i \in \overline{S}$ corresponds to a monomial containing at least one of the $x_i$, so they all correspond to annihilators of $m$ in $\overline{M}$, and $|NA(m)| = 0$. 
    
    If $n$ is odd, we let $m = x_0^{d-1}\cdots x_{n-1}^{d-2} y^{\frac{n+1}{2}}$. Elements of $NA(m)$ must correspond to monomials of the form $x_{n-1}^ay^b$ with $a = 1$ and $b \leq \frac{d}{2}-\frac{n+1}{2} -1$. The maximum degree of such an element is $1 + d-n+1 - 2 < d$, so this element cannot correspond to a generator of $\overline{S}$. Thus, $|NA(m)| = 0$ in this case as well. 

    Moreover, $B_n(d) = N-n$ regardless of the parity of $n$.  
\end{proof}

The following lemma is the analogue of of Lemma \ref{lem:oddOverlap} for $d$ even. 
\begin{lemma}
\label{lem:evenOverlap}
    For $d$ even and $1 \leq q \leq n$, $\beta_{i, i+q} \neq 0$ for all $F_q(d) \leq i \leq B_q(d)$. 
\end{lemma}
\begin{proof}
    For a row $q$, let $m_1$ denote the monomial used to compute $F_q(d)$ and $m_2$ denote the monomial used for $B_q(d)$. There are three possible $m_1$, $m_2$ pairs, so we consider the cases where $q = 1$, $2 \leq q \leq n-1$, and $q = n$ separately. 

    \textbf{Row 1:} As in Lemma \ref{lem:oddOverlap}, for the first row, $m_1 = m_2$, so there is no overlap to check. 

    \textbf{Rows} $\bm{2\leq q \leq n-1 \colon}$ For these rows, $m_1 = x_0^{d-1}\cdots x_{q-2}^{d-1}x_{q-1}^{q-1}y^{\frac{d}{2}-1}$ and $m_2 = x_0^{d-1}\cdots x_{q-1}^{d-1}x_q^{q}$. First, note that the degree $d$ monomials in $R_{q+1,0}/I := k[x_0, \ldots, x_q] /\langle x_0^d, \ldots, x_{q-1}^d, x_q^{q+1}\rangle$ are in bijection with $D(m_2)$. In Lemma \ref{lem:oddOverlap}, we computed $\operatorname{Hilb}(d, R_{q+1,0}/I)$ for precisely this ring, so we won't repeat it here. We have 
    \[|D(m_2)| = {d+q \choose q} - { d-1 \choose q} - q =  b''d^q + (\text{lower order terms}),\]
    where $b''$ is some constant.

    Next, observe that $|NA(m_2)| = \operatorname{Hilb}(d, R_{n-q+2, 0}/I')$ for $R_{n-q+2, 0} =k[x_{q-1}, \ldots, x_{n-1}]$ and $I' = \langle x_{q-1}^{d-q-1}, x_q^d, \ldots, x_{n-1}^d \rangle$. After applying an alternating sum of Hilbert functions to the minimal free resolution, we have 
    \[|NA(m_1)| = \operatorname{Hilb}(d, R_{n-q+2, 0}) = \operatorname{Hilb}(q+1, R_{n-q+2, 0}) - (n-q).\] 
    The ring $R_{n-q+2, 0}$ is standard graded, so 
    \[|A(m_1)| = N - n - |NA(m)| =  N - n - {d+n-q+1 \choose n-q+1} + {n+2 \choose n-q+1} + {n-q}.\]

    Therefore, for $d \gg 0$, 
    \[|A(m_1)| = ad^{n} + (\text{lower order terms})- \left(a'd^{n-q+1} +(\text{lower order terms})\right),\]
    for some constants $a$ and $a'$.
    Since $q >1$, $|A(m_1)|$ is asymptotically $d^n$. Lastly, since $q \leq n-1$, 
    \[ |D(m_2)| = b''d^q+ (\text{lower order terms}) \leq ad^{n} + (\text{lower order terms}) = |A(m_1)|.\]

    \textbf{Row} $\bm{n\colon}$ For the last row, $m_1 = x_0^{d-1} \cdots x_{n-2}^{d-1}x_{n-1}^{n+1} y^{\frac{d}{2}-1}$. Observe that $m_1$ is annihilated by all generators of $\overline{S}$, so that $|A(m_1)| = N-n \geq |D(m_2)|$ by construction.
\end{proof}

\subsection{Main Theorems}\hfill
\label{subsec:MainThrm}

Theorem \ref{Thrm:A} and Theorem \ref{thrm:preciseBoundsIntro} summarize the previous four sections. We provide proofs of the main theorems below.

\begin{Bproof}
        This theorem is an amalgamation of Lemmas \ref{lem:F1d}, \ref{lem:Fqd}, \ref{lem:genBqd}, \ref{lem:secondLastBqd}, \ref{lem:lastBqd}, \ref{lem:oddOverlap}, \ref{lem:evenFqd}, \ref{lem:evenBqd1}, \ref{lem:EvenBqd2}, and \ref{lem:evenOverlap}. 
    \end{Bproof}

\begin{Aproof}
    Following from Theorem \ref{thrm:preciseBoundsIntro} and Remark \ref{rmk:asypN}, it suffices to prove the following facts:
\begin{enumerate}
 \item \label{Aproof:1} When d is odd and $2\leq q \leq n+1$ then $F_q(d)$ agrees with a period 2 polynomial of degree $\leq {q-2}$ for $d\gg 0$. 
 \item When d is odd and $1\leq q \leq n-1$ then $N - B_q(d)$ agrees with a period 2 polynomial of degree $\leq {n-q}$ for $d\gg 0$.
\item  When d is even and $2\leq q \leq n$ then $F_q(d)$ agrees with a period 2 polynomial of degree $\leq {q-1}$ for $d\gg 0$.
\item  When d is even and $1\leq q \leq n-1$ then $N - B_q(d)$ agrees with a period 2 polynomial of degree $\leq {n-q}$ for $d\gg 0$.
\end{enumerate}
For (\ref{Aproof:1}): Rearranging the formula for $F_q(d)$ given in Lemma \ref{lem:Fqd} yields that
\[F_q(d) +(q-2) = \operatorname{Hilb}(d,R_{q-1,1}) - \operatorname{Hilb}(d-q-1,R_{q-1,1}) = \operatorname{Hilb}(d, R_{q-1,1}/f),\]
where $f$ is some degree $q+1$ homogeneous element in $R_{q-1,1}$. Since $q-2$ is a constant in $d$, the statement follows from Part (\ref{cor:asympN:hypersurface}) of Lemma \ref{lem:asympN}, with $i = q-1$.

For (2):  Lemma \ref{lem:genBqd} shows that $N - B_q(d) +\operatorname{Hilb}(q,R_{n-q,1})-q-1 =\operatorname{Hilb}(d,R_{n-q,1})$.  Since $\operatorname{Hilb}(q,R_{n-q,1})$ and $-q-1$ are constants in $d$, the statement follows from Part (\ref{cor:asympN:poly}) of Lemma \ref{lem:asympN}.

For (3):  Lemma \ref{lem:evenFqd} shows that $F_q(d) + q = \operatorname{Hilb}(d, R_{n-q,1}/g)$ for $g$ some degree $q+2$ homogeneous element in $R_{n-q,1}$. The statement then follows from Part (\ref{cor:asympN:hypersurface}) of Lemma \ref{lem:asympN}.

For (4): Lemma \ref{lem:evenBqd1} shows that $N - B_q(d) + \operatorname{Hilb}(q, R_{n-q,1}) - q = \operatorname{Hilb}(d, R_{n-q,1})$, so $N-B_q(d)$ is, up to a constant, equal to $\operatorname{Hilb}(d,R_{n-q,1})$, and the statement follows from Part (\ref{cor:asympN:poly}) of Lemma \ref{lem:asympN}.
\end{Aproof}

\begin{corollary}
    We continue with the hypotheses of Theorem \ref{Thrm:A}, and let $M$ be the coordinate ring of the Veronese. For $\rho_q(M)$ defined as in \cite{ErmanYang_2018} and the introduction, we have
    \[\rho_q(M) = \begin{cases}
        1 & \text{if } 1 \leq q \leq n\\
        1 & \text{if } q=n+1 \text{ and } d \text{ is odd} \\
        0 & \text{else }
    \end{cases}.\]
\end{corollary}

\begin{proof}
    Recall that, for $d \gg 0$, $N$ is on the same order of magnitude as $d^n$ (Remark \ref{rmk:asypN}). Therefore, for every row in Theorem \ref{Thrm:A}, the column index of the left-most guaranteed nonzero entry is at least one order of magnitude less than the right-most entry. This allows us to conclude that, asymptotically, $\rho_q(M) = 1$ for $q$ in the allowable range. The regularity arguments in Section \ref{section:regularity} show that the Betti table is zero for any row outside of this range, meaning that $\rho_q(M) = 0$ otherwise.  
\end{proof}

\section{Challenges for Other Weighted Projective Spaces}
\label{section:futureDir}

To conclude, we will highlight some additional challenges that arise when considering other weighted projective spaces. 

In Section \ref{section:regularity}, we showed that computing the number of rows in the Betti table gets increasingly more complicated as the degrees of the variables become larger (and that each setting splinters into cases, depending on the relationship between $d$ and each of the variables). Example \ref{ex:P(1,1,3)Reg} highlights this complication by exploring Veronese embeddings of $\mathbb{P}(1,1,3)$.

For this section, we will focus on a different challenge that arises when studying the asymptotic syzygies of other weighted projective spaces. One key feature of $\mathbb{P}(1^n,2)$ is that we can perform an Artinian reduction that yields a finite module generated by elements of degree 1. However, it's possible to be left with a finite module that has $k$-algebra generators with degree $\neq 1$ after the Artinian reduction. These generators of higher degree make both selecting the correct monomials and determining their syzygies much more nuanced. We explore this through the following example. 

\begin{example}
\label{ex:futureDir}
     Let $R = k[x_0, x_1, y]$ so that $\Proj(R) = \mathbb{P}(1,1,3)$, and let $S = k[z_0, \ldots,z_{11}]$, so that the $z_i$'s correspond to the generators of $M = R^{(5)}$, in the order listed below: 
    \[M = k[x_0^5, \; x_1^5, \; \bm{y^5}, \; x_0^4x_1, \; x_0^3x_1^2, \; x_0^2x_1^3, \; x_0x_1^4, \; x_0^2y, \; x_0x_1y, \; x_1^2y, \; \bm{x_0y^3}, \; \bm{x_1y^3}].\]

    Notice that the monomials in bold have degree greater than to 1 (in the grading inherited from $S$). In particular, $\deg(y^5) = 3$ and $\deg(x_0y^3) = \deg(x_1y^3) = 2$. We perform an Artinian reduction to get $\overline{M} = M/ \langle x_0^5, x_1^5, y^5 \rangle$, an $\overline{S} = S/\langle z_0, z_1, z_2 \rangle$-module. The key difference here is that $\overline{S}$ is no longer generated in degree 1, as it has two variables of degree 2 ($z_{10}$ and $z_{11}$). Crucially, since $y^5$, $x_0y^3$, and $x_1y^3$ do not form a regular sequence, no choice of Artinian reduction will allow us to simultaneously remove all three, and we will always be left with some $z_i$ that are not in $\overline{S}_1 $. 
    
    For $\mathbb{P}(1^n,2)$ it was sufficient to find homology elements in the strand 
    \[\wedge^{i-1}\overline{S}_1 \otimes \overline{M}_{j+1} \longleftarrow \wedge^i \overline{S}_1 \otimes \overline{M}j \longleftarrow \wedge^{i+1} \overline{S}_1 \otimes \overline{M}_{j-1};\]  
    however, when $\overline{S}$ is not generated in degree 1, we need to allow the $z_i \not\in \overline{S}_1$ to appear in the wedge product as well. If we let $\overline{m}$ be the maximal ideal generated by the variables in $\overline{S}$, then we actually wish to find nonzero homology elements in the strand 
    \[\wedge^{i-1}\overline{m}/\overline{m}^2 \otimes \overline{M}_{j+1} \longleftarrow \wedge^i \overline{m}/\overline{m}^2 \otimes \overline{M}j \longleftarrow \wedge^{i+1} \overline{m}/\overline{m}^2 \otimes \overline{M}_{j-1}.\]  
    
    This example is again small enough that we can compute the Betti table in Macaulay2, shown below. 

        \begin{center}
        \begin{tabular}{ c | c c c c c c c c c c c  } 
         & 0 & 1 & 2 & 3 & 4 & 5 & 6 & 7 & 8 & 9   \\
        \hline
        0 & 1 & - & - & - & - & - & -& - & -& -  \\ 
        1 & - & 21 & 70 & 105 & 84 & 35 & 6 & - & - & -  \\ 
        2 & - & 14 & 84 & 210 & 280 & 210 & 84 & 14 & - & - \\ 
        3 & - & 9 & 63 & 189 & 315 & 315 & 189 & 63 & 9 & - \\
        4 & - & - & 14 & 84 & 210 & 280 & 210 & 84 & 14  & - \\
        5 & - & - & - & 6 & 35 & 84 & 105 & 70 & 21 & - \\
        6 & - & - & - & - & - & -& -& -& -& 1 
        \end{tabular}
    \end{center}

    We see that this table extends through row 6, as anticipated by Example \ref{ex:P(1,1,3)Reg}. However, the the highest degree monomial in $\overline{M}$ is $m = x_1^4x_2^4y^4$, which has degree $4$. In the previous arguments, $\deg(m)$ gave the index of the row where that monomial yields nonzero entries; however, this is only true in examples where $\overline{S}$ is generated in degree 1. In this setting, $m$ will yield the nonzero Betti entry at $\beta_{9,15}$, but we have to track the syzygies very carefully to see this. 

    Note that $m$ is annihilated by every generator of $\overline{S}$, so we expect it to give a nonzero entry in column 9 via the term 
    \[z_3 \wedge \ldots \wedge z_{10} \wedge z_{11} \otimes m.\]
    However, since $z_{10}$ and $z_{11}$ each have degree 2, each one bumps the nonzero Betti entry down a row from its expected location. Thus, $m$ will actually correspond to a nonzero Betti entry in column 9 of row $\deg(m) + 2 = 6$ (namely, $\beta_{9,15}$). 

    An interesting phenomenon can be seen here: when an element is divisible by or annihilated by a generator with degree $\neq 1$, that generator pushes the nonzero block down a row in the Betti table. The upshot is that individual monomials now yield nonzero blocks in the shape of parallelograms, spanning multiple rows in the Betti table! 

    To further emphasize this point, consider the monomial $m = x_0^4y^2$. We can define $D(m)$ and $A(m)$ as in Section \ref{section:notation}, so that 
    \[ D(m) = \{z_7\} \text{ and } A(m) = \{z_3, z_4, z_5, z_6, z_7, z_8, z_{10}, z_{11} \},\] as the only non-annihilator is $z_9$ which corresponds to $x_1^2y$. 

    Since $\deg(z_7) = 1$, the element $z_7 \otimes m$ will yield a nonzero Betti entry in row $\deg(m) = 2$ of the Betti table. We can add the other degree 1 elements of $A(m)$ to the wedge product to get a nonzero block corresponding to $m$ that stays in row 2. The last such element is 
    \[ \left( z_3 \wedge z_4 \wedge z_5 \wedge z_6 \wedge z_7 \wedge z_8\right)  \otimes m \implies \beta_{6,8} \neq 0,\]
    so that $m$ gives nonzero entries in row 2 between columns 1 and 6. 
    
    However, if we instead consider $\left(z_7 \wedge z_{10} \right) \otimes m$, we get a nonzero entry in row 3 of the Betti table (since $\deg(z_{10}) = 2$). We can similarly add the degree 1 elements of $A(m)$ to this wedge product to get a nonzero block of entries in the third row. The last such element is 
    \[\left(z_3 \wedge \ldots \wedge z_8 \wedge z_{10} \right) \otimes m \implies \beta_{7, 10} \neq 0, \]
    so $m$ also gives a block of nonzero entries in row 3 between columns 2 and 7. 

    Lastly, we could consider $\left( z_7 \wedge z_{10} \wedge z_{11} \right) \otimes m$, which yields $\beta_{3,7} \neq 0$. We can add the remaining elements of $A(m)$ to the wedge product to get a block of nonzero entries in row 4 between columns 3 and 8. 

    Putting this all together, we see that the single monomial $m$ accounts for all of the Betti entries highlighted below. 

    \begin{center}
       \begin{tabular}{ c | c c c c c c c c c c c  } 
         & 0 & 1 & 2 & 3 & 4 & 5 & 6 & 7 & 8 & 9   \\
        \hline
        0 & 1 & - & - & - & - & - & -& - & -& -  \\ 
       1 & - & 21 & 70 & 105 & 84 & 35 & 6 & - & - & -  \\ 
        2 & - & \hl{14} & \hl{84} & \hl{210} & \hl{280} & \hl{210} & \hl{84} & 14 & - & - \\ 
        3 & - & 9 & \hl{63} & \hl{189} & \hl{315} & \hl{315} & \hl{189} & \hl{63} & 9 & - \\
        4 & - & - & 14 & \hl{84} & \hl{210} & \hl{280} & \hl{210} & \hl{84} & \hl{14}& - \\
       5 & - & - & - & 6 & 35 & 84 & 105 & 70 & 21 & - \\
        6 & - & - & - & - & - & -& -& -& -& 1 
        \end{tabular}
   \end{center}
   
    This entire Betti table is recoverable by applying the EEL Method to well-chosen monomials. However, unlike the $\mathbb{P}(1^n,2)$ setting, we may need to consider more than two monomials per row. For example, 
    row $3$ of the Betti table sees pieces of the nonzero blocks corresponding to three monomials, where each is the lex-most monomial in some degree with respect to an ordering on the variables. In particular, $x_0^4y^2$ gives the highlighted entries above, $x_0y^3$ yields that $\beta_{1,4} \neq 0$, and $x_0^4x_1^3y$ gives $\beta_{8,11} \neq 0$. This is possible now that a monomial need not have degree $3$ in order to yield nonzero Betti entries in row $3$! 
\end{example}

Beyond $\mathbb{P}(1^n,3)$, these complications get exponentially worse. For example, both $\mathbb{P}(1^n,2)$ and $\mathbb{P}(1^n,3)$ admit only two distinct term orders on the corresponding variables (up to symmetry). However, for an arbitrary weighted projective space, there may be numerous distinct orders where each yields a distinct monomial to consider in every degree.

\bibliographystyle{alpha}
\bibliography{bibliography}

\end{document}

%% file: preamble.tex
\usepackage{times, amsthm, amssymb, amsmath, amsfonts, graphicx, mathrsfs, comment}
\usepackage{mathtools}
\usepackage{tikz} 
\usetikzlibrary{arrows, cd, matrix, positioning}
\usepackage{xcolor}
\usepackage{bbm}
\usepackage{xypic}
\usepackage{mathscinet}
\usepackage{hyperref}
\usepackage{wrapfig}
\usepackage{subcaption}
\usepackage{color,soul}
\usepackage{bm}
\usepackage{booktabs}

\usepackage{graphicx}
\usepackage{tabularray}

\newtheorem{theorem}{Theorem}[section]
\newtheorem{lemma}[theorem]{Lemma}

\newtheorem{corollary}[theorem]{Corollary}

\newtheorem{introtheorem}{Theorem}

\newtheorem{introconjecture}[introtheorem]{Conjecture}

\newtheorem{introcor}[introtheorem]{Corollary}

\theoremstyle{definition}
\newtheorem{definition}[theorem]{Definition}
\newtheorem{example}[theorem]{Example}

\newenvironment{Aproof}{%
\proof}{\endproof}

\newenvironment{Bproof}{%
\proof}
{\endproof}

\theoremstyle{remark}
\newtheorem{remark}[theorem]{Remark}

\newcommand{\rank}{\operatorname{rank}}

\newcommand{\Hilb}{\operatorname{Hilb}}
\newcommand{\Proj}{\operatorname{Proj}}
\newcommand{\reg}{\operatorname{reg}}

\newcommand{\Manoa}{M\=anoa}

\newcommand{\Hawaii}{Hawai\kern.05em`\kern.05em\relax i}

\DeclarePairedDelimiter\floor{\lfloor}{\rfloor}
